\documentclass[12pt,leqno]{amsart}
\usepackage[notref,notcite]{showkeys} % show labels
\usepackage{amsmath,stmaryrd}
\usepackage{amssymb}
\usepackage{amsthm}
\usepackage{epsf}
\usepackage{graphicx}
\usepackage{esint} 
\usepackage{dsfont}
\usepackage[pagebackref]{hyperref} 
\hypersetup{pdfpagemode=FullScreen,  colorlinks=true} 
\usepackage{enumerate} 
\usepackage{xcolor,bbm,easybmat,bm}
\usepackage{accents}

\usepackage{amsfonts}
\usepackage{amsrefs}
\usepackage{verbatim}
\usepackage{booktabs}
\usepackage{color}
\usepackage{graphicx,float}

\numberwithin{equation}{section}

\makeatletter
\newcommand\footnoteref[1]{\protected@xdef\@thefnmark{\ref{#1}}\@footnotemark}
\makeatother

\newcommand{\vv}{\accentset{\approx}{v}}

\newcommand{\vp}{\varphi}
\newcommand{\dr}{\partial}

\DeclareMathOperator{\divg}{div}

\DeclareMathOperator{\dist}{dist}
\DeclareMathOperator{\supp}{supp}
\DeclareMathOperator{\diam}{diam}

\DeclareMathOperator{\loc}{loc}

\newcommand{\1}{{\mathds 1}}

\newcommand{\R}{\mathbb R}

\newcommand{\cF}{\mathcal F}
\renewcommand{\H}{\mathcal H}

\newcommand{\wt}{\widetilde}

\newcommand{\D}{\mathbb D}

\renewcommand{\L}{L}

\newcommand{\A}{\mathcal A}
\newcommand{\C}{\mathcal C}

\newcommand{\OO}{\mathcal O}

\newcommand{\W}{\mathcal W}

\newcommand{\Rn}{\mathbb R^n}

\newcommand{\abs}[1]{\left\vert#1\right\vert}
\newcommand{\br}[1]{\left(#1\right)}
\newcommand{\avg}[1]{\langle#1\rangle}
\newcommand{\set}[1]{\left\{#1\right\}}
\renewcommand{\div}{\mathrm{div}}
\renewcommand{\d}{\, \mathrm{d}} %differential

\newcommand{\om}{\Omega}
\newcommand{\pom}{\partial\Omega}

\newcommand{\E}{\mathsf{E}} %Energy space
\newcommand{\Lloc}{\L_{\operatorname{loc}}} %local Lebesgue spaces
\newcommand{\HT}{H_t} %Hilbert transform
\newcommand{\dhalf}{D_t^{1/2}} %half time derivative
\newcommand{\Hdot}{\dot{H}\protect{\vphantom{H}}} %homogeneous Sobolev spaces
\newcommand{\mS}{{\mathcal S}} %Schwartz space/Single layer
\newcommand{\IC}{\mathbb{C}}
\newcommand{\pd}{\partial}
\newcommand{\cl}[1]{\overline{#1}} %closure

\newcommand{\ree}{{\mathbb{R}^{n}}}

\newcommand{\lpformath}{\rm{(}}
\newcommand{\rpformath}{\rm{)}}

\newcommand{\Nformath}{\rm{N}}
\newcommand{\Dformath}{\rm{D}}

\newcommand{\wPNformath}{\rm{wPN}}

\newcommand{\wPRNformath}{\rm{wPNR}}
\newcommand{\PDformath}{\rm{PD}}
\newcommand{\PRformath}{\rm{PR}}

\newcommand{\Np}{\hyperlink{Np}{$\lpformath\Nformath_{p}\rpformath$}}
\newcommand{\Nq}{\hyperlink{Np}{$\lpformath\Nformath_{q}\rpformath$}}

\newcommand{\wPNq}{\hyperlink{wPNq}{$\lpformath\wPNformath_{p'}\rpformath$}}

\newcommand{\wPRNp}{\hyperlink{wPRNp}{$\lpformath\wPRNformath_{p}\rpformath$}}

\newcommand{\Dq}{\hyperlink{Dq}{$\lpformath\Dformath_{p'}\rpformath$}}

\newcommand{\PDq}{\hyperlink{PDq}{$\lpformath\PDformath_{p'}\rpformath$}}
\newcommand{\PRp}{\hyperlink{PRp}{$\lpformath\PRformath_{p}\rpformath$}}

\newcommand{\dint}{\int\!\!\!\!\!\int}
\def\Yint#1{\mathchoice
	{\YYint\displaystyle\textstyle{#1}}%
	{\YYint\textstyle\scriptstyle{#1}}%
	{\YYint\scriptstyle\scriptscriptstyle{#1}}%
	{\YYint\scriptscriptstyle\scriptscriptstyle{#1}}%
	\!\dint}
\def\YYint#1#2#3{{\setbox0=\hbox{$#1{#2#3}{\iint}$}
		\vcenter{\hbox{$#2#3$}}\kern-.51\wd0}}
\def\longdash{\mkern-1.5mu{-}\mkern-7.5mu{-}} 
\def\fiint{\Yint\longdash}

\newcommand{\Z}{{\mathbb Z}}

\theoremstyle{plain}
\newtheorem{theorem}[equation]{Theorem}
\newtheorem{lemma}[equation]{Lemma}

\newtheorem{definition}[equation]{Definition}

\theoremstyle{definition}

\theoremstyle{remark}
\newtheorem{remark}[equation]{Remark}

\begin{document}

\title{Localization and interpolation of parabolic $L^p$ Neumann problems}

\author[Dindo\v{s}]{Martin Dindo\v{s}}
\address{School of Mathematics, 
The University of Edinburgh and Maxwell Institute of Mathematical Sciences, Edinburgh, UK}
\email{M.Dindos@ed.ac.uk}

\author[Li]{Linhan Li}
\address{School of Mathematics, The University of Edinburgh and Maxwell Institute of Mathematical Sciences, Edinburgh, UK}
\email{linhan.li@ed.ac.uk}

\author[Pipher]{Jill Pipher}
\address{Department of Mathematics, 
 Brown University, RI, US}
\email{jill\_pipher@brown.edu}

%\thanks{The first author has been supported in part by EPSRC grant EP/Y033078/1.}

\maketitle

\begin{abstract} 
We show a localization estimate for local solutions to the parabolic equation $-\dr_t u+\divg (A\nabla u)=0$ with zero Neumann data, assuming that the $L^p$ Neumann problem and $L^{p'}$ Dirichlet problem for the adjoint operator are solvable in a Lipschitz cylinder for some $p\in(1,\infty)$. Using this result, we establish the solvability of the Neumann problem in the atomic Hardy space for parabolic operators with bounded, measurable, time-dependent coefficients, and hence obtain the extrapolation of solvability of the $L^p$ Neumann problem.   
\end{abstract}

\tableofcontents

%\ms\noindent{\bf Keywords:} 

%\ms\noindent
%\subjclass{2020 Mathematics Subject Classification: 35K20, 35K10}

\section{Introduction}
The main objective of this paper is to establish a localization result for a parabolic Neumann problem for equations of the form
$L u= -\dr_t u +\divg (A \nabla u)=0$. The coefficients of the elliptic matrix $A$ are assumed to be bounded, measurable and are allowed to be time varying. In both the elliptic and parabolic settings, such an estimate is key to developing a complete theory of $L^p$ solvability of boundary value problems. Indeed, we go on to show (Section \ref{S5}) how our estimate can be used for interpolation of Neumann solvability, i.e., assuming the solvability of the $L^p$ parabolic Neumann problem (and also its adjoint $L^{p'}$ Dirichlet problem)
one can obtain solvability of the $L^q$ parabolic Neumann problem for all $1<q\leq p$.
The localization estimate says that, on the portion of the Neumann boundary where the datum vanishes, the nontangential maximal function of the gradient $\nabla u$ has an upper bound in terms of  $\nabla u$ just on an enlarged Carleson region.  Localization
in the context of Neumann boundary
conditions
has previously only been established for certain special operators such as the heat operator using different methods (where solvability on subdomains was used).\medskip

More precisely, with appropriate assumptions on the 
$L^p$ solvability of a parabolic Neumann problem,
we seek to establish the estimate
\begin{equation} \label{loc1x}
\|\wt N(|\nabla u|\1_{J_r})\|_{L^p(\partial \Omega)} \leq C r^{(n+1)/p}\left ( \fiint_{J_{2r}\cap \om} |\nabla u|^2 dXdt \right)^\frac12,
\end{equation}
for a backward in time parabolic cube $J_r$ centered at a boundary point such that the Neumann data $\partial^A_\nu u=0$ in $J_{2r}\cap\partial\Omega$. Here $Lu=0$, where
\begin{equation}\label{E:pde}
			L u:= -\dr_t u +\divg (A \nabla u)=0   \quad\text{in } \Omega, 
\end{equation}
and the domain $\Omega = \mathcal O \times \R$ is a bounded or unbounded Lipschitz cylinder, i.e., $\mathcal O = \R^n_+$ or $\mathcal O$ is a bounded or unbounded Lipschitz domain. We assume that  $A= [a_{ij}(X, t)]$ is a $n\times n$ bounded matrix satisfying the uniform ellipticity condition for a.e. $X \in \mathcal O$, $t\in \R$.
That is, there exist positive constants $\lambda$ and $\Lambda$ such that
\begin{equation}
	\label{E:elliptic}
	\lambda |\xi|^2 \leq \sum_{i,j} a_{ij}(X,t) \xi_i \xi_j,\qquad |A(X,t)| \leq \Lambda,
\end{equation}
for almost every $(X,t) \in \Omega$ and all $\xi \in \R^n$.\medskip

Estimates of the form \eqref{loc1x} for both elliptic and parabolic PDEs have multiple uses and serve as a tool for establishing Hardy space atom estimates for both Regularity and Neumann boundary value problems. When instead of the vanishing Neumann boundary condition we assume vanishing Dirichlet condition, the estimate is rather easy to prove under the assumption that the elliptic or (parabolic) measure of the adjoint PDE belongs to the $A_\infty(d\sigma)$ class. In this case the estimate holds for $p\in (1,p_0)$,
where $p_0$ is the value such that for all $q>p_0'$ the $L^q$ Dirichlet problem for the adjoint operator $L^*$ is solvable. See for example \cite{DK} treating the elliptic case when the boundary is Lipschitz, or \cite{DiS} for the setting of parabolic PDEs on a Lipschitz cylinder. 

The parabolic Dirichlet localization result has been used in, for example, the paper \cite{DiS} where interpolation and atomic estimates for the Regularity problem were established, or \cite{DLP} where the Regularity problem for the parabolic PDE is solved in the optimal range on a Lipschitz cylinder 
under the assumption that the coefficients of the PDE satisfy certain natural Carleson conditions (a parabolic analogue of so-called DKP-condition). There are also a number of earlier papers making use of localization estimates to show Hardy space estimates in the context of Dirichlet, Regularity or Neumann boundary data.
\medskip 

The first and main objective of this paper is to prove the following theorem.
\begin{theorem}\label{thm.NtoLoc}
Let $\om=\mathcal O\times\R$ where $\mathcal O$ is either a bounded or unbounded Lipschitz domain in $\Rn$. Let $L=-\dr_t +\divg (A\nabla \cdot)$, and let $L^*=\dr_t+\divg(A^T\nabla\cdot)$ be the adjoint operator of $L$.
Let $p\in (1,\infty)$. Suppose that \Np$^L$ and \Dq$^{L^*}$ are solvable in $\om$ (see Definitions~\ref{def.Dq}, \ref{def.Np}).
    Let $Q_0:=Q_0'\times(T_0,T_1)$ be a cube centered at $\pom$. Then for any $(\bar x,\bar t)\in Q_0\cap\pom$, any backward parabolic cube $J_r(\bar x,\bar t)$ that satisfies $J_{2r}(\bar x,\bar t)\subset Q_0$, and any weak solution $u$ to $Lu=0$ in $Q_0\cap\om$ with zero Neumann data on $Q_0\cap\pom$, we have
\begin{equation*}
\|\wt N(|\nabla u|\1_{J_r(\bar x,\bar t)})\|_{L^p(\partial \Omega)} \leq C r^{(n+1)/p}\left ( \fiint_{J_{2r}(\bar x,\bar t)\cap \om} |\nabla u|^2 dXdt \right)^\frac12.
\end{equation*}
\end{theorem}
\begin{remark}
     A note on the proof strategy: It might seem natural to follow the original strategy taken in \cite[Theorem 6.10]{KP93} which established Neumann localization in the elliptic setting. The proof of their result required assuming solvability of both the $L^p$ Neumann and Regularity problems, and using the Regularity result in a key fashion. However, in the parabolic case, we wish to avoid working via the $L^p$ Regularity problem as it involves a half-derivative in time on the boundary data, making localization estimates significantly more difficult to obtain. Fortunately, an analogous elliptic Neumann localization result for more general boundaries was established in \cite{FL} (along with its implications), via a method that dispensed with the assumption on the Regularity problem. In \cite{FL}, they achieved the desired localization estimate by introducing and studying the Poisson–Neumann problem and its weak variants. This is
     the strategy we employ here, introducing certain analogous parabolic problems and adapting the arguments of \cite{FL} to the time dependent parabolic setting.
\end{remark}
\smallskip

The second objective of this paper is to present an important application of the localization estimate, formulated in the following theorem.

\begin{theorem}\label{MT} Let $\Omega=\mathcal O\times\R$, where 
$$\mathcal O=\{(x,x_n): x_n>\varphi(x)\},\qquad\mbox{for }x\in\R^{n-1},$$
for a Lipschitz function $\varphi:\R^{n-1}\to\R$. Consider the PDE 
$$L u= -\dr_t u +\divg (A \nabla u)=0   \quad\text{in } \Omega.$$ 

Assume that for some $p\in(1,\infty)$ the $L^p$ Neumann problem \Np$^L$  and the $L^{p'}$ Dirichlet problem  \Dq$^{L^*}$ for the adjoint PDE $L^*u=0$ are both solvable (as defined in Section \ref{S2}).  
 
 Then for all $1<q<p$ the $L^q$ Neumann problem \Nq$^L$ is also solvable. 
\end{theorem}

 For specific operators (such as the heat equation on Lipschitz cylinders, or even more general domains), this result is known (cf.\ \cite{B}), but here it is established for operators with no particular assumptions on coefficients beyond ellipticity. Hence this result is more analogous to previously established elliptic theory for the Neumann problem as in \cite{KP93}. Further, we expect that this interpolation result will be very
useful for studying the Neumann problem for classes of specific operators. One such example is our paper in preparation \cite{DLP2}, applying the main result here to a subclass of operators of the type considered in \cite{DLP} and showing optimal $L^p$ solvability of the Neumann problem for such PDEs when the Carleson measure condition on the coefficients is assumed to be sufficiently small.
 
 \section{Definitions}\label{S2}

 \begin{definition}
We define a parabolic cube on $\R^n\times\R$  centered at $(X,t)$ with sidelength $r$ as
$$    Q_r(X,t):=\{ (Y, s) \in \R^{n}\times\R : |x_i - y_i| < r \ \text{ for } 1 \leq i \leq n, \ | t - s |^{1/2} < r \},
$$
and we write $\ell(Q_r)=r$.
When writing a lower case point $(x,t)$ we shall mean a boundary parabolic cube on $\R^{n-1}\times\R$
which has an analogous definition but in one less spatial dimension:
\begin{equation}\label{eqdef.bdypcube}
    Q_r(x,t):=\{ (y, s) \in \R^{n-1}\times\R : |x_i - y_i| < r \ \text{ for } 1 \leq i \leq n-1, \ | t - s |^{1/2} < r \}.
\end{equation}
We use the notation $Q'_r(X)$ to denote the cube in $\Rn$ centered at $X\in\Rn$ with sidelength $r$, unless specified otherwise.

For $(X,t)\in \R^n\times\R$, we denote by $J_r(X,t)$ the backward in time parabolic cube, that is,
\begin{equation}\label{eqdef.bwpcube}
   J_r(X,t): =Q_r'(X)\times (t-r^2,t). 
\end{equation}

A parabolic ball on $\R^n\times\R$  centered at $(X,t)$ with radius $r$ is the ball
\begin{equation}\label{eqdef.ball}
    B_r(X,t):=\{ (Y, s) \in \R^{n}\times\R : d_p((X,t),(Y,s))<r \},
\end{equation}
where $d_p(\cdot,\cdot)$ is the parabolic distance function
\[
d_p((X,t),(Y,s)) := \br{\abs{X-Y}^2+\abs{t-s}}^{1/2}.\]
%\sim \|(X-Y,t-s)\|.\]

For parabolic balls at the boundary we use notation $\Delta_r(X,t)=B_r(X,t)\cap \pom$. In the special case $\Omega=\R^n_+\times\R$ we drop the last coordinate and also write $\Delta_r(x,t)$ with understanding that the ball is centered at $(X,t)=(x,0,t)$.
\end{definition}

\begin{definition}
For $a>0$ and $(q,\tau)\in\pom$, unless otherwise defined, we denote the non-tangential parabolic cones by
\begin{equation}\label{Gamma2.11}
    \Gamma_a(q,\tau):=\set{(X,t)\in\om: d_p((X,t),(q,\tau))<(1+a)\delta(X,t)},
\end{equation}
and $\delta(\cdot)$ is the parabolic distance to the boundary: 
\[
\delta(X, t) = \inf_{(q,\tau)\in\pom}
d_p((X, t),(q, \tau)).\]
We also use the truncated parabolic cones: for $r>0$,
\[
\Gamma_a^r(q,\tau):=\set{(X,t)\in\om: d((X,t),(q,\tau))<(1+a)\delta(X,t),\, \delta(X,t)<r}
\]
is the parabolic cone with vertex $(q,\tau)$ truncated at height $r$.
\end{definition}

\subsection{Tent spaces}

\begin{definition}
For $w\in L^\infty_{\loc}(\om)$, we define the non-tangential maximal function of $w$ as
\[
 N_a(w)(q,\tau):=\sup_{(X,t)\in\Gamma_a(q,\tau)}\abs{w(X,t)} \quad\text{for }(q,\tau)\in\pom,
\]
the parabolic area functional  
\[
\mathcal{A}^a_1(w)(q,\tau):=\iint_{\Gamma_a(q,\tau)}|w(X,t)|\frac{dXdt}{\delta(X,t)^{n+2}},
\]
and the parabolic Carleson functional
\[
\mathcal{C}_1(w)(q,\tau):=\sup_{r>0}\frac1{\sigma(B_r(q,\tau)\cap\pom)}\iint_{B_r(q,\tau)\cap\om}|w(X,t)|\frac{dXdt}{\delta(X,t)},
\]
where $\sigma = \H^{n+1}_p|_{\pom}$ is the $(n+1)$-dimensional parabolic Hausdorff measure\footnote{For a set $E \subset \mathbb{R}^{n+1}$, the $(n+1)$-dimensional parabolic Hausdorff measure is defined as $\H^{n+1}_{p}(E) := \lim_{\epsilon \to 0^+} \H^{n+1}_{p,\epsilon}(E)$, where \(\H^{n+1}_{p,\epsilon}(E) := \inf\left\{\sum_i \diam(E_i)^s: E \subseteq \cup_i E_i, \diam(E_i) \le \epsilon \right\},\) and $\diam(E):=\sup_{(X,t),(Y,s)\in E}d_p((X,t),(Y,s))$.
} restricted to $\pom$.
If $w\in L^p_{\loc}(\om)$\footnote{with respect to the $(n+1)$-dimensional Lebesgue measure}, $p\in(0,\infty)$, we define for $c\in(0,1)$ the modified non-tangential maximal function
\begin{equation}\label{def.Nap}
    \wt N_{a,p}^c(w)(q,\tau):=\sup_{(X,t)\in\Gamma_a(q,\tau)}\br{\fiint_{Q_{c\delta(X,t)/2}(X,t)}\abs{w(Y,s)}^pdYds}^{1/p} \quad\text{for }(q,\tau)\in\pom,
\end{equation}
the modified parabolic area functional 
\[
\mathcal{\wt A}^{a,c}_{1,p}(w)(q,\tau):=\iint_{\Gamma_a(q,\tau)}\br{\fiint_{Q_{c\delta(X,t)/2}(X,t)}\abs{w(Y,s)}^pdYds}^{1/p}\frac{dXdt}{\delta(X,t)^{n+2}},
\]
and the modified Carleson functional 
\[
\mathcal{\wt C}^c_{1,p}(w)(q,\tau):=\sup_{r>0}\frac1{\sigma(B_r(q,\tau)\cap\pom)}\iint_{B_r(q,\tau)\cap\om}\br{\fiint_{Q_{c\delta(X,t)/2}(X,t)}\abs{w(Y,s)}^pdYds}^{1/p}\frac{dXdt}{\delta(X,t)}.
\]
We simply denote 
\begin{equation}\label{def.N2}
    \wt N(w):=\wt N_{1,2}^{1/2}(w), \quad \mathcal{\wt A} (w):=\mathcal{\wt A}_{1,2}^{1,1/2} (w), \quad \mathcal{\wt C}:=\mathcal{\wt C}_{1,2}^{1/2}. 
\end{equation}
\end{definition}

The following lemma states that the $L^p$ norms remain comparable---up to multiplicative constants---when the apertures and the sizes of the Whitney cubes are changed in the definitions of the modified non-tangential maximal function, the area functional, and the Carleson functional. In the isotropic setting, these facts are well known. For instance, the invariance under changes of aperture follows from the argument in \cite[Proposition 4]{CMS}; the effect of changing the Whitney cube sizes is treated in \cite{HR13} and also in \cite[Lemma 2.2]{MPT} for domains with more general boundaries. Moreover, the equivalence \eqref{eq.C=A} follows from the argument in \cite[Theorem 3]{CMS}. These arguments extend to the parabolic setting with only minor modifications, and we therefore omit the details.
\begin{lemma} \label{lemNCA}
Let  $a >1$, $c\in (0,1)$, and $p\in (0,\infty]$. Then 
\[\|\wt N^{c}_{a,2}(u)\|_{L^p(\partial \Omega,\sigma)} \approx \|\wt N(u)\|_{L^p(\partial \Omega,\sigma)},\]
\begin{equation}\label{eq.Aac=A}
    \|\wt \A_{1,2}^{a,c}(u)\|_{L^p(\partial \Omega,\sigma)} \approx \|\wt \A(u)\|_{L^p(\partial \Omega,\sigma)},
\end{equation}
and
\[\|\wt \C_{1,2}^{c}(u)\|_{L^p(\partial \Omega,\sigma)} \approx \|\wt \C(u)\|_{L^p(\partial \Omega,\sigma)},\]
where the constants involved depend on $\Omega$, $a$, $c$ and $p$. Moreover, if $p\in (1,\infty)$, then
\begin{equation}\label{eq.C=A}
    \|\wt \C(u)\|_{L^p(\partial \Omega,\sigma)} \approx \|\wt \A(u)\|_{L^p(\partial \Omega,\sigma)},
\end{equation}
where the constants depend on $\Omega$ and $p$.
\end{lemma}

With those in hand, the space of functions $u$ for which $\|\wt N^{c}_{a,2}(u)\|_{p}$, $\|\wt \A^{a,c}_{1,2}(u)\|_{p}$ or $\|\wt \C_{1,2}^{c}(u)\|_{p}$ are finite doesn't depend on $a$ or $c$. We call them ($L^2$ averaged) tent spaces $\wt T^p_\infty(\Omega)$ and $\wt T^p_1(\Omega)$ whose norms are
\begin{align}
&\|u\|_{\wt T^p_\infty(\Omega)} := \|\wt N(u)\|_{L^p(\partial \Omega,\sigma)},\nonumber\\
&\|u\|_{\wt T^p_1(\Omega)} := \|\wt \A(u)\|_{L^p(\partial \Omega,\sigma)},\qquad\text{for }p\in(0,\infty),\quad
\|u\|_{\wt T^\infty_1(\Omega)} := \|\wt \C(u)\|_{L^\infty(\partial \Omega,\sigma)}.\nonumber
\end{align}
%{\lh I don't want to follow referee's suggestion on this one since for $p<\infty$, $\|\cdot\|_{\wt T^p_1(\Omega)}$ is already defined.}

The next lemma gives duality between the modified Carleson functional and non-tangential maximal function. In the isotropic setting, these results are obtained in \cite{HR13} when $\om$ is the upper half-space and in \cite[Proposition 2.4]{MPT} when the domain is more general. In both works, they considered duality between $\wt N_{a,p}$ and $\wt\C_{1,p'}$ for $p>1$. The argument carries over to the parabolic settings, and since we only need the $L^2$ averages in the modified non-tangential maximal function and the Carleson functional, we state the result only for $\wt N$ and $\wt\C$.
\begin{lemma}[Duality between $\wt\C$ and $\wt N$] If  $p\in [1,\infty)$, then for any $u\in \wt T^p_\infty(\Omega)$ and any $F$ that satisfies  $\delta F\in \wt T^{p'}_1(\Omega)$,
\begin{equation}\label{eq.CNdual}
    \abs{\iint_{\om}uF\,dXdt}\lesssim \|\wt\C(\delta F)\|_{L^{p'}(\pom)}\|\wt N(u)\|_{L^p(\pom)}.
\end{equation}
Moreover
\begin{equation}\label{eq.NpC}
    \|\wt N(u)\|_{L^p(\pom)}\lesssim \sup_{F:\|\wt\C(\delta F)\|_{L^{p'}(\pom)}=1}\abs{\iint_{\om}uF\,dXdt},
\end{equation}
and 
\begin{equation}\label{eq.NpC2}
    \|\wt\C(u)\|_{L^{p'}(\pom)}\lesssim \sup_{F:\|\wt N(\delta F)\|_{L^{p}(\pom)}=1}\abs{\iint_{\om}uF\,dXdt}.
\end{equation}
\end{lemma}

\subsection{Reinforced weak and Energy solutions}\label{RwEs}

We recall the paper \cite{AEN} that neatly presents the concept of reinforced weak solutions for the parabolic problem of interest here. In \cite{Din23}, it was shown that the definition given in \cite{AEN} can be weakened a bit further by only asking for the \lq\lq local $1/2$ derivative" in the time variable. We explain below.
\vglue1mm

If $\mathcal O$ is an open subset of $\mathbb R^{n} $, we let $\H^1(\mathcal O)=\W^{1,2}(\mathcal O)$ be the standard Sobolev space of real valued functions $v$ defined on $\mathcal O$, such that $v$ and $\nabla v$ are in $\L^{2}(\mathcal O;\R)$ and $\L^{2}(\mathcal O;\R^n)$, respectively. A subscripted `$\loc$' will indicate that these conditions hold locally.

We shall say that $u$ is a \emph{reinforced weak solution} of $-\partial_t u + \divg(A\nabla u)=0$ on $\Omega=\mathcal O\times \R$ if
\begin{align*}
 u\in \dot {\E}_{\loc}(\Omega):= \H_{\loc}^{1/2}(\R; \L^2_{\loc}(\mathcal O)) \cap \Lloc^2(\R; \W^{1,2}_{\loc}(\mathcal O))
\end{align*}
and if for all $\phi,\psi \in \C_0^\infty(\Omega)$,
\begin{equation}\label{2.10}
\iint_{\Omega}\left[
 A\nabla u\cdot{\nabla (\phi\psi)}+ \HT\dhalf (u\psi)\cdot {\dhalf \phi}+ \HT\dhalf (u\phi)\cdot {\dhalf \psi}\right]\, \d X \d t=0.
 \end{equation}
Here, $\dhalf$ is the half-order derivative and $\HT$ is the Hilbert transform with respect to the $t$ variable, normalized so that $\partial_{t}= \dhalf \HT \dhalf$. The space $\Hdot^{1/2}(\R)$ is the homogeneous Sobolev space of order 1/2 - the completion of $\C_0^\infty(\R)$ in the norm $\|\dhalf (\cdot)\|_{2}$ -  and it embeds into the space $\mS'(\R)/\IC$ of tempered distributions, modulo constants.
The local space  $\H_{\loc}^{1/2}(\R)$ consists of functions $u$ such that $u\phi\in \Hdot^{1/2}(\R)$ for all $\phi\in \C_0^\infty(\Omega)$. As shown in \cite{Din23} the space $\H_{\loc}^{1/2}(\R)$ is larger than $\Hdot^{1/2}(\R)$.
For  $u\in\Hdot^{1/2}(\R; \L^2_{\loc}(\mathcal O)), $ \eqref{2.10} simplifies to 
\begin{equation}\nonumber
\iint_{\Omega}\left[
 A\nabla u\cdot{\nabla \phi}+ \HT\dhalf u\cdot {\dhalf \phi}\right]\, \d X \d t=0.
 \end{equation}
  
 Our definition has the advantage that, in taking a cut-off of the function $u$,  we might potentially improve the decay of $D^{1/2}_tu$ at infinity. 

At this point we remark that for any $u\in \Hdot^{1/2}(\R)$ and $\phi,\psi\in \C_0^\infty(\R)$ the formula
\begin{align*}
 \int_{\R} \left[\HT\dhalf (u\psi)\cdot {\dhalf\phi}+\HT\dhalf (u\phi)\cdot {\dhalf\psi}\right]\d t = - \int_{\R} u \cdot {\partial_{t}(\phi\psi)} \d t
\end{align*}
holds, where on the right-hand side we use the duality form extension of the complex inner product of $\L^2(\R)$, between $\Hdot^{1/2}(\R)$ and its dual $\Hdot^{-1/2}(\R)$. By taking $\psi=1$ on the set where $\phi$ is supported, it follows that a reinforced weak solution is a weak solution in the usual sense on $\Omega$ since it satisfies
$u\in \Lloc^2(\R; \W^{1,2}_{\loc}(\mathcal O))$ and for all $\phi\in \C_0^\infty(\Omega)$,
\begin{align}\label{eq.weaksol}
 \iint_{\Omega} A\nabla u\cdot{\nabla \phi} \d X \d t - \iint_{\Omega} u \cdot {\partial_{t}\phi} \d X  \d t=0.
\end{align}
 This implies $\pd_{t}u\in \Lloc^2(\R; \W^{-1,2}_{\loc}(\mathcal O))$. Conversely, any  weak solution $u$ in    $ \H_{\loc}^{1/2}(\R; \L^2_{\loc}(\mathcal O))$ is a reinforced weak solution.\vglue2mm

Specialising to the case $\Omega=\mathcal O\times\mathbb R$, where $\mathcal O$ is either a bounded or an unbounded Lipschitz domain,
 we say that a reinforced weak solution $v\in \dot{\E}_{\loc}(\mathcal O\times\mathbb R)$ belongs to the \emph{energy class} $\dot \E(\mathcal O\times\mathbb R)$ if
\begin{align*}
 \|v\|_{\dot \E} := \bigg(\|\nabla v\|_{\L^2(\mathcal O\times\mathbb R)}^2 + \|\HT \dhalf v\|_{\L^2(\mathcal O\times\mathbb R)}^2 \bigg)^{1/2} < \infty.
\end{align*}
Consequently, these are called \emph{energy solutions}. When considered modulo constants, $\dot \E $ is a Hilbert space and it is in fact the closure of $\C_0^\infty\!\big(\,\cl{\mathcal O\times\mathbb R}\,\big)$ for the homogeneous norm $\|\cdot\|_{\dot \E}$
and it coincides with the space $\dot{L}^2_{1,1/2}(\Omega)$.

As shown in \cite{AEN} (with a small generalization), functions from $\dot \E $ have well defined Dirichlet traces with values in
 the \emph{homogeneous parabolic Sobolev space} $\Hdot^{1/4}_{\pd_{t} - \Delta_x}(\partial\mathcal O\times\mathbb R)$. Here, $\Hdot^{s}_{\pm \pd_{t} - \Delta_x}(\R^n)$ is defined as the closure of Schwartz functions $v \in \mS(\ree)$ with Fourier support away from the origin in the norm $\|\cF^{-1}((|\xi|^2 \pm i \tau)^s \cF v)\|_2$. This yields a space of tempered  distributions modulo constants in $\Lloc^2(\ree)$ if $0 < s \leq 1/2$. 
 Conversely, any $g \in \Hdot^{1/4}_{\pd_{t} - \Delta_x}$ can be extended to a function $v \in  \dot \E$ with trace $v\big|_{\partial\mathbb R^{n+1}_+} = g$.  From this via partition of unity we can define 
 $\Hdot^{1/4}_{\pd_{t} - \Delta_x}(\partial\mathcal O\times\mathbb R)$ for $\mathcal O$ Lipschitz.
  
Hence, by the energy solution to $-\partial_tu + \divg(A\nabla u)=0$ with Dirichlet boundary datum $u\big|_{\partial\mathcal O\times\R} = f \in \Hdot^{1/4}_{\pd_{t} - \Delta_x}$ (understood in the trace sense) we mean $u \in \dot\E$ such that
\begin{align*}
a(u,v):= \iint_{\mathcal O\times\R} \left[A \nabla u \cdot{\nabla v} + \HT \dhalf u \cdot {\dhalf v}\right] \d X \d t = 0,
\end{align*}
holds for all $v \in \dot \E_0$, the subspace of $\dot \E$ with zero boundary trace.

Moving onto the Neumann problem, given any $u \in \dot \E(\Omega)$, the co-normal derivative $\partial^A_\nu u\Big|_{\partial\Omega}=:\langle A\nabla u,\nu\rangle \Big|_{\partial\Omega}=g$ is defined via the formula
\begin{align}\label{eq.NeumannBdy}
 \iint_{\mathcal O\times\R} \left[A \nabla u \cdot{\nabla v} + \HT \dhalf u \cdot {\dhalf v}\right] \d X \d t -\int_{\partial\mathcal O
 \times\R}gv \d x\d t= 0,
\end{align}
for all $v \in \dot \E$. Here, since the traces of $v$ belong to $\Hdot^{1/4}_{\pd_{t} - \Delta_x}(\partial\Omega)$ and all elements of the space $\Hdot^{1/4}_{\pd_{t} - \Delta_x}$ are realized by some $v \in \dot \E$,
 the Neumann boundary data must by duality  naturally belong to the space $\Hdot^{-1/4}_{\pd_{t} - \Delta_x}(\partial\Omega)$.

By \cite{AEN}, the key to solving these problems is the introduction of the modified sesquilinear form (introduced earlier in \cite{Ny2016}):
\begin{equation}\label{eq-sesq}
 a_\delta(u,v) := \iint_{\mathcal O\times\R} \left[A \nabla u \cdot {\nabla (1-\delta \HT) v} + \HT \dhalf u \cdot {\dhalf (1-\delta \HT) v}\right] \d X \d t,
\end{equation}
where $\delta$ is a  real number yet to be chosen. The Hilbert transform $\HT$ is a skew-symmetric isometric operator with inverse $-\HT$ on both $\dot \E$ and $\Hdot^{1/4}_{\pd_{t} - \Delta_x}$.  Hence, $1-\delta \HT$ is invertible on these spaces for any $\delta \in \R$. Hence for a fixed $\delta>0$  small enough, $a_\delta$ is coercive on $\dot \E$ since
\begin{equation}\label{eq:coer}
 a_\delta(u,u) \ge (\lambda-\Lambda\delta )\|\nabla u\|_2^2 + \delta \|\HT \dhalf u \|_2^2.
\end{equation} 
In particular $\delta=\lambda/(\Lambda+1)$ would work.
To solve the Dirichlet problem we take an extension $w \in \dot \E$ of the data $f$ and apply the Lax-Milgram lemma to $a_\delta$ on $\dot \E_0$ to obtain some $u \in \dot \E_0$ such that 
\begin{align*}
 a_\delta(u,v) = - a_{\delta}(w,v) \qquad (v \in \dot \E_0).          \end{align*}
Hence, $u + w$ is an energy solution with data $f$. Should there exist another solution $v$, then $a_\delta(u+w-v,u+w-v) = 0$ and hence by coercivity $\|u +w - v \|_{\dot \E} = 0$. Thus the two solutions only differ by a constant. Thus the  Dirichlet problem associated with our parabolic PDE is  \emph{well-posed} in the energy class. Similar arguments allow us to solve the Neumann problem by considering for the datum $g \in \Hdot^{-1/4}_{\pd_{t} - \Delta_x}(\partial\Omega)$ the solution $u
\in \dot{E}$ such that 
\begin{align*}
 a_\delta(u,v) = \langle g, \mbox{Tr } (1-\delta \HT)v\rangle\qquad (v \in \dot \E).          \end{align*}

\begin{definition}[\Dq]\label{def.Dq}\hypertarget{Dq}{}
    Let $p\in(1,\infty)$. We say that the $L^{p'}$ Dirichlet problem is solvable for $L^*$, denoted by \Dq or \Dq$^{L^*}$, if there exists a constant $C>0$ such that for all $g \in \Hdot^{1/4}_{\pd_{t} - \Delta_x}\cap L^p(\pom)$, the energy solution $u\in \dot E(\om)$ to $L^* u=\partial_tu+\divg(A^T\nabla u) = 0$ in $\om$ with trace $u|_{\pom}=g$ satisfies the estimate
    \[
    \|N(u)\|_{L^p(\pom)}\le C\|g\|_{L^p(\pom)}.
    \]
\end{definition}
\begin{definition}[\Np]\label{def.Np}\hypertarget{Np}{}
     Let $p\in(1,\infty)$. We say that the $L^{p}$ Neumann problem is solvable for $L$, denoted by \Np or \Np$^{L}$, if there exists a constant $C>0$ such that for all $g \in \Hdot^{-1/4}_{\pd_{t} - \Delta_x}(\partial\Omega)\cap L^p(\pom)$, the energy solution $u\in \dot E(\om)$ to $Lu = 0$ in $\om$ with $\dr_\nu^Au|_{\pom}=g$ (defined as in \eqref{eq.NeumannBdy}) satisfies the estimate
    \[
    \|\wt N(u)\|_{L^p(\pom)}\le C\|g\|_{L^p(\pom)}.
    \]
\end{definition}
 
 \medskip

\noindent{\bf Energy solutions for the Poisson problem.}\label{ESPP} We now address the existence of energy solutions for the inhomogeneous PDE
$$Lu=-\dr_t u +\divg (A \nabla u)= h-\divg F,$$
with zero Dirichlet or Neumann data. 

To avoid any boundary terms arising from the pairing of the right-hand side with a test function we shall assume that $h$ is a bounded function with bounded support and $F$ is a bounded function with compact support in $\Omega$.

Let $v\in \dot \E$  be a test function. Then the pairing $\langle -\divg F,v\rangle$ can be understood in the sense
$$\langle -\divg F,v\rangle=\iint_{\mathcal O\times\R} F\cdot\nabla v\,dXdt,$$
and since $\nabla v\in L^2(\Omega)$ we need $\|F\|_{L^2(\Omega)}<\infty$, which certainly any bounded compactly supported $F$ satisfies. The above defines a bounded functional on $\dot \E_0$ for the Dirichlet problem and on $\dot \E$ for the Neumann problem and thus we can solve both problems with a right-hand side of the form $\divg F$.\medskip

Considering the right-hand side for the term $h$ for the Dirichlet problem is also not problematic but we will need to understand the embedding of the space $\dot\E_0$ into the $L^p$ spaces. It is stated in \cite{AEN} Remark
3.13 (without proof) that the space $\dot\E(\R^{n+1})/\mathbb C$ has realization in the Lebesgue space $L^q(\R^{n+1})$ for $q=2(n+2)/n$. From this (using an odd extension across the flat boundary) we get that $\dot\E_0(\R^n_+\times\R)$
embeds into $L^q(\R^n_+\times\R)$. Using a partition of unity, the embedding of $\dot\E_0(\mathcal O\times\R)$
into $L^q(\mathcal O\times\R)$ follows. Hence it remains to see why the original claim as stated 
in \cite{AEN} holds. Taking the same approach as in Tao's notes \cite[Exercise 41]{T} it suffices to understand
into which $L^q(\R^{n+1})$ belongs the convolution of an $L^2$ function with the kernel
$$\hat{f}(\xi,\tau)=C_n\rho(\xi,\tau)^{-(n+1)},$$
where $ \|(\xi,\tau)\| = \rho$, for $\xi\in\mathbb R^n$ and $\tau\in\mathbb R$
and $\rho$ is the positive solution  to the following equation:
$$   \frac{|\xi|^2}{\rho^2} + \frac{\tau^2}{\rho^4} = 1.$$ 
It holds that $\rho\sim |\xi|+ |\tau|^{1/2}$, and that $f$ is the function $\rho(x,t)^{-1}=\|(x,t)\|^{-1}$. Observe that the $\dot\E(\R^{n+1})$ norm can be defined
by
	\begin{equation*}
		\|g\|_{\dot\E(\R^{n+1})} = \|\D g\|_{L^2(\R^{n+1})}=\left\| \|(\xi,\tau)\|\hat g(\xi,\tau)\right\|_{L^2(\R^{n+1})},
	\end{equation*}
    where 
    \begin{equation*}
		(\D f)\,\widehat{\,}\,(\xi,\tau) := \|(\xi,\tau)\| \widehat{f}(\xi,\tau).
	\end{equation*}
This then implies that $g=h*\hat f$ for some $L^2$ function $h$. 	
 Analysis of the kernel $\hat f$ above reveals that it scales in $\xi$ as $|\xi|^{-(n+1)}$ but in $\tau$ as $|\tau|^{-(n+1)/2}$. The final conclusion
requires understanding of fractional integrals with kernels of mixed homogeneity (see Chapter 1 for the parabolic norm structure and Theorem 6.1 for the fractional integration result of \cite{FS}). This reveals that the resulting convolution belongs to the strong $L^q(\R^{n+1})$ space for $q=2(n+2)/n$, as desired.\medskip

As $h$ is bounded with  bounded
support, it follows that 
 the pairing $\langle h,v\rangle$ enjoys the bound
$$|\langle h,v\rangle|=\left|\iint_{\mathcal O\times\R} hv\,dXdt\right|\lesssim \|h\|_{L^{q'}}\|v\|_{\dot\E_0},$$
and therefore defines a bounded linear functional on $\dot \E_0$. This again  implies the existence of a solution for the Poisson Dirichlet problem with right-hand side containing the term $h$. We do not consider the Neumann problem for the term $h$.\medskip

\noindent{\bf More on weak solutions.}\label{STestF}
Let $u$ be a weak solution to $Lu=0$ in $\om$, that is, $u\in \L_{\loc}^2(\R; \W^{1,2}_{\loc}(\mathcal O))$, 
$\pd_{t}u\in \Lloc^2(\R; \W^{-1,2}_{\loc}(\mathcal O))$, and \eqref{eq.weaksol} holds for all $\phi\in C_c^\infty(\om)$. We can rewrite 
\eqref{eq.weaksol} as
\begin{equation}\label{1.3}
 \iint_{\Omega} A\nabla u\cdot{\nabla \phi} \d X \d t + \int_{\R}\langle\partial_t u, \phi\rangle  \d t=0,
\end{equation}
where the pairing $\langle \partial_t u, \phi\rangle $ is understood as pairing of $\W^{-1,2}_{\loc}(\mathcal O)$
with $\W^{1,2}_{0}(\mathcal O)$ function subsequently integrated in time. We refer to solutions $u$ such that 
$u\in \L_{\loc}^2((t_0,t_1); \W^{1,2}_{\loc}(\mathcal O))$ and 
$\pd_{t}u\in \Lloc^2((t_0,t_1); \W^{-1,2}_{\loc}(\mathcal O))$ when the support of $\phi$
is limited to the set $\mathcal O\times [t_0,t_1]$.

 Let $\theta_\varepsilon:\R\to\R$ be a family of standard mollifiers, that is, $\theta_\varepsilon(\tau)=\frac{1}{\varepsilon}\theta(\tau/\varepsilon)$ for $\tau\in\R$, $\varepsilon>0$, where $\theta\in C_c^\infty(-1,1)$, $\int_\R \theta=1$, and $0\le \theta\le 1$. 
Let $\psi\in C_c^\infty(\mathcal O)$ be a fixed function. For a fixed $t\in(t_0,t_1)$, we consider the smooth function  $\phi_\varepsilon(X,\tau):=\psi(X)\theta_\varepsilon(t-\tau)$. Then for a weak solution $u$ on $\OO\times (t_0,t_1)$,  thanks to \eqref{1.3}
we have that
\begin{equation}\label{1.4}
 \iint_{\Omega}A(X,\tau)\nabla u(X,\tau)\cdot{\nabla \psi(X)}\,\theta_\varepsilon(t-\tau) \d X \d \tau 
 + \int_\R\langle \partial_\tau u(\cdot,\tau), \psi\rangle\, \theta_\varepsilon(t-\tau)  \d \tau=0.
\end{equation}
As $\varepsilon\to  0_+$, the first term of \eqref{1.4} converges to 
$$
\int_{\mathcal O}  A(X,t)\nabla u(X,t)\cdot{\nabla \psi(X)}\d X \qquad\text{for a.e. }t\in (t_0,t_1).
$$
For the second term, since $\dr_t u\in  \Lloc^2((t_0,t_1); \W^{-1,2}_{\loc}(\mathcal O))$, the function $\tau\mapsto\langle \partial_\tau u(\cdot,\tau), \psi\rangle$ is in $L_{\loc}^2(t_0,t_1)$, and therefore, for a.e. $t\in(t_0,t_1)$,
$$
\int_\R\langle \partial_\tau u(\cdot,\tau), \psi\rangle\, \theta_\varepsilon(t-\tau) \d \tau
\to
\langle \partial_\tau u(\cdot,t), \psi\rangle \qquad\text{as }\varepsilon\to 0_+.
$$
Hence for a.e. $t\in (t_0,t_1)$, we have 
\begin{equation}\label{1.5}
\langle \partial_\tau u(\cdot,t), \psi\rangle
%_{\W_{\loc}^{-1,2}(\OO),\W_0^{1,2}(\OO)}
+\int_{\mathcal O}  A(X,t)\nabla u(X,t)\cdot{\nabla \psi(X)}\d X=0 \qquad\text{for any }\psi\in C_c^\infty(\mathcal O),
\end{equation}
and equivalently (by a density argument), for any $\psi\in \W_0^{1,2}(\mathcal O)$.
In particular, it implies that if $u$ is a weak solution to $Lu=0$ in $\OO\times(t_0,t_1)$, then 
\begin{equation}\label{eq.wksol_cutoff}
\int_{t_0}^{t_1}\langle \partial_\tau u(\cdot,t), \phi(\cdot,t)\rangle\,dt +\int_{t_0}^{t_1}\int_{\mathcal O}  A\nabla u\cdot{\nabla \phi}\d Xdt=0\quad\text{for any }\phi\in L_{\loc}^2((t_0,t_1);\W_0^{1,2}(\OO)),
\end{equation}
since we can take $\psi=\phi(\cdot,t)$ in \eqref{1.5} and integrate in $t$ over $(t_0,t_1)$.
The observations \eqref{1.5} and \eqref{eq.wksol_cutoff} will be very useful since they allow us to take sharp cut-offs in time.

We remark that similar observations also apply to weak solutions (and therefore, energy solutions) to the Poisson problem $Lu=-\div F$ with zero Dirichlet and Neumann boundary data; we skip the details.

% Conversely, if \eqref{1.5} holds for a.e. $t\in (t_0,t_1)$, then \eqref{1.3} follows by using $\phi=\varphi(\cdot,t)$
% and integrating in $t$ over $(t_0,t_1)$. Hence \eqref{1.3} and \eqref{1.5} are equivalent for
% $u$ that satisfies $u\in \L_{\loc}^2(\R; \W^{1,2}_{\loc}(\mathcal O))$ and 
% $\pd_{t}u\in \Lloc^2(\R; \W^{-1,2}_{\loc}(\mathcal O))$.

\medskip

\subsection{The $L^p$ Poisson problems}
Let $L=-\dr_t +\divg (A\nabla \cdot)$, and $L^*=\dr_t+\divg(A^T\nabla\cdot)$ be the adjoint operator of $L$. 
\begin{definition}[\PDq,\PRp]\hypertarget{PDq}{}
    Let $p\in(1,\infty)$. We say that the $L^{p'}$ Poisson-Dirichlet problem is solvable for $L^*$, denoted by \PDq or \PDq$^{L^*}$, if there exists a constant $C>0$ such that for all $F\in L_c^\infty(\om,\Rn)$, the energy solution $u$ to $L^* u=\divg F$ in $\om$ with $u=0$ on $\pom$ satisfies
    \[
    \| N(u)\|_{L^{p'}(\pom)}\le C\|\wt\C(\delta F)\|_{L^{p'}(\pom)}.
    \]
    We say that the $L^{p}$ Poisson-Regularity problem is solvable for $L$, denoted by \PRp or \PRp$^{L}$, \hypertarget{PRp}{}if there exists a constant $C>0$ such that for all $h\in L_c^\infty(\om)$, the energy solution $u$ to $L u=h+\divg F$ in $\om$ with $u=0$ on $\pom$ satisfies
    \[
    \| \wt N(\nabla u)\|_{L^{p}(\pom)}\le C\br{\| \wt \C(\delta h)\|_{L^{p}(\pom)}+\|\wt \C(F)\|_{L^p(\pom)}}.
    \]
\end{definition}

\begin{definition}\label{def.PN} \hypertarget{PNq}{}
Let $p\in (1,\infty)$. 
%We say that the $L^{p'}$ Poisson-Neumann problem \PNq\ is solvable for $L^*$ if for any ${F}\in L^\infty_c(\Omega,\R^n)$, the energy solution $u$ to $L^*u=-\divg F$ with zero Neumann boundary data satisfies
% \[
% \|\wt N(u)\|_{L^{p'}(\partial \Omega)} 
% %\leq C \|\delta  F\|_{\wt T^{p'}_1} := 
% \le C  \|\wt\C(\delta{F})\|_{L^{p'}(\partial \Omega)},\]
% where $C$ is independent of $F$. 
We say that the weak $L^{p'}$ Poisson-Neumann problem\footnote{We follow the terminology introduced in \cite{FL} and therefore refer to these problems as {\it weak}. Since the stronger versions are not needed for the purposes of this paper, we do not define them.} \wPNq\ is solvable for $L^*$ if for any ${F}\in L^\infty_c(\Omega,\R^n)$, the energy solution $u$ to $L^*u=-\divg F$ with zero Neumann boundary data satisfies 
\[
\|\wt N(\delta \nabla u)\|_{L^{p'}(\partial \Omega)} \leq C  \|\wt\C(\delta{F})\|_{L^{p'}(\partial \Omega)},
\]
where $C$ is independent of $F$. 
We write %\PNq$_{L^*}$ or 
\wPNq$_{L^*}$ when we want to mention the operator.
\end{definition}

\begin{definition} \hypertarget{wPRNp}{}
Let $p\in (1,\infty)$. 
% We say that the $L^p$ Poisson-Neumann-Regularity \PRNp\ is solvable for $L$ if for any $h\in L^\infty_c(\Omega)$ and ${F}\in L^\infty_c(\Omega,\R^n)$, the solution $u$ to $Lu=h-\divg F$ with 0 Neumann boundary data satisfies
% \[\|\wt N(\nabla u) \|_{L^{p}(\partial \Omega)} \leq C \|\wt\C_1(\delta |h| + |{F}|)\|_{L^{p}(\partial \Omega)},\]
% where $C$ is independent of $h$ and $ F$. 
We say that the weak $L^p$ Poisson-Neumann-Regularity \wPRNp\ is solvable for $L$ if for any ${F}\in L^\infty_c(\Omega,\R^n)$, the solution $u$ to $Lu=-\divg F$ with 0 Neumann boundary data satisfies
\[
\|\wt N(\nabla u) \|_{L^{p}(\partial \Omega)} \leq C  \|\wt\C({F})\|_{L^{p}(\partial \Omega)},
\]
where $C$ is independent of $F$.
We write 
%\PRNp$_L$ or 
\wPRNp$_L$ when we want to mention the operator.
\end{definition}

\section{Basic estimates}
We assume $\om=\mathcal O\times\R$ where $\mathcal O$ is a Lipschitz domain. 
% \begin{lemma}[Caccioppoli's inequality, see~\cite{Aro68}]%
% 	\label{io}
% 	Let $A$ satisfy \eqref{E:elliptic} and  suppose that $u$ is a weak solution of \eqref{E:pde} in $Q_{4r}(X,t)\subset\Omega$.
% 	Then there exists a constant $C=C(\lambda,\Lambda, n)$ such that
% 	\begin{multline*}
% 	    r^{n} \left(\sup_{Q_{r/2}(X,t)} u \right)^{2}
% 			\leq C 
%    \sup_{t-r^2 \leq s \leq t+r^2} \int_{Q_r(X,t) \cap \{t = s\}} u^{2}(Y,s) dY\\
% 			+ C\iint_{Q_{r}(X, t)} |\nabla u|^{2} dYds 
% 			\leq \frac{C^2}{r^2} \iint_{Q_{2r}(X, t)} u^{2}(Y, s) dYds.
% 	\end{multline*}
% \end{lemma}

\begin{lemma}[Caccioppoli's inequality]\label{L:Caccio}
Let $Q_0\subset\R^{n+1}$ be a cube satisfying $Q_0\cap\om\neq\emptyset$.  Let $u$ be a weak solution of $Lu=0$ in $Q_0\cap\om$. If $Q_0\cap\pom\neq\emptyset$, then we assume in addition that $u$ has zero Neumann data on ${Q_0}\cap\pom$.

Let $(\bar X,\bar t)\in Q_0\cap\overline{\om}$ and let $J_r:=J_r(\bar X,\bar t)$ be a backward in time parabolic cube that satisfies $J_{4r}\subset Q_0$. 
Then there exists a constant $C=C(\lambda,\Lambda, n)$ such that for any constant $c$,
	\[
	    \iint_{J_{r}\cap\om} |\nabla u|^{2} dYds 
			\leq \frac{C}{r^2} \iint_{J_{2r}\cap\om} |u(Y,s)-c|^{2} dYds.
	\]
\end{lemma}
\begin{proof}
    When $Q_0\subset\om$, it is the standard interior Caccioppoli inequality whose proof is classical. %{\lh I can't find a reference for this, but proving it is easy}. 
    When $Q_0\cap\pom\neq\emptyset$, we assume without loss of generality that it is centered at $\pom$.
    %and that $(\bar X,\bar t)\in Q_0\cap\pom$. 
    Assume that $\mathcal O\cap Q_0=\set{(x,x_n)\in\Rn: x_n>\vp(x)}\cap Q_0$ for some Lipschitz function $\vp$, and consider the even reflection of $u$ defined as 
    \begin{equation}\label{eqdef.reflec}
    \tilde u(x,x_n,t)=\begin{cases}
u(x,x_n,t),&\qquad\mbox{for }x_n\ge \varphi(x),\\
u(x,2\varphi(x)-x_n,t),&\qquad\mbox{for }x_n< \varphi(x).\end{cases}   
   \end{equation}
   Since $u$ has zero Neumann data on $Q_0\cap\pom$, $\tilde u$ is a weak solution to $\tilde L\tilde u=-\dr_t \tilde u+\divg(\tilde A\nabla \tilde u)=0$ in $Q_0$, where $\tilde A=A$ on $\Omega$ and the coefficient matrix $\tilde A$ on the set $x_n<\varphi(x)$ depends only on $A(x,2\varphi(x)-x_n,t)$ and $\nabla\varphi$ making the resulting matrix bounded and uniformly elliptic on the set $Q_0\setminus\om$. Applying the interior Caccioppoli's inequality to $\tilde u$ gives
   \[
   \iint_{J_{r}}|\nabla\tilde u|^{2} dYds 
			\leq \frac{C}{r^2} \iint_{J_{2r}} |\tilde u(Y,s)-c|^{2} dYds,
   \]
   which yields the desired estimate since
   \[
   \iint_{J_{r}\cap\om}|\nabla u|^2dYds\lesssim \iint_{J_r}|\nabla\tilde u|^2dYds, \text{ and }
    \iint_{J_{2r}} |\tilde u(Y,s)-c|^{2} dYds\lesssim
     \iint_{J_{2r}\cap\om} |u(Y,s)-c|^{2} dYds.
   \]
\end{proof}

\begin{lemma}\label{lem.MoserNeu}
Let $Q_0$, $J_r$, and $u$ be as in Lemma~\ref{L:Caccio}.
    %Let $(x,t)\in\pom$. Let $u$ be a weak solution to \eqref{E:pde} in $Q_{4r}(x,t)\cap\om$ with zero Neumann data on $Q_{4r}(x,t)\cap\pom$. 
    There exists a constant $C=C(\lambda,\Lambda,n)$ such that
    \begin{equation}\label{eq.localbdd}
          \sup_{J_{r/2}\cap\om}|u| \le \frac{C}{r^{n+2}}\iint_{J_{2r}\cap\om}|u(Y,s)|dYds.
    \end{equation}
  \end{lemma}
\begin{proof}
   When $Q_0\subset\om$, i.e., in the interior case, the result is classical. See for example \cite[Theorem B]{Aro68}, which gives \eqref{eq.localbdd} with an $L^2$ version on the right-hand side. Going from $L^2$ to $L^1$ is also classical and is based on a dilation argument. 
When $Q_0\cap\pom\neq\emptyset$, we use the reflection \eqref{eqdef.reflec} to reduce matters to the interior case. We skip the details.
\end{proof}

\begin{lemma}[Poincar\'e inequality for solutions]\label{lem.Pnc_sol} 
Let $Q_0=Q_0'\times (T_0,T_1)$ be a cube satisfying $Q_0\cap\om\neq\emptyset$. 
Let $u$ be a weak solution of $Lu=-\div  F$ in $Q_0\cap\om$, and assume in addition that $u$ has zero Neumann data on $Q_0\cap\pom$ if $Q_0\cap\pom\neq\emptyset$.
Let $(X,t)\in Q_0\cap\overline{\om}$ and let $J_r:=J_r(X,t)$ be a backward in time parabolic cube that satisfies $J_{4r}\subset Q_0$.  

Then there is some constant $C$ such that 
\[
\iint_{J_r\cap\om}|u-c_u|dXdt\le C r\iint_{J_{2r}\cap\om}\br{|\nabla u|+|  F|}dXdt,
\]
where $c_u=\fiint_{J_r\cap\om}u\,dXdt$.
\end{lemma}
\begin{proof}
    We only prove the case when  $Q_0\cap\pom\neq\emptyset$, as the interior case is similar. We can assume without loss of generality that $(0,0)\in Q_0\cap \pom$ and $J_r=Q'\times (-r^2,0)$, where $Q'=\set{X\in\Rn: |x_i|<r}$ is the cube in $\Rn$ centered at $0$ with sidelength $r$. For $t\in(T_0, T_1)$, we set $\avg{u}(t):=\fint_{Q'\cap\OO}u(Y,t)dY$. 
    By the triangle inequality and the Poincar\'e inequality in the spatial variables, 
    \begin{multline}\label{eq.u-cu}
         \iint_{J_r\cap\om}|u(X,t)-c_u|dXdt\le  \iint_{J_r\cap\om}|u(X,t)-\langle u\rangle(t)|dXdt+\iint_{J_r\cap\om}|\avg{u}(t)-c_u|dXdt\\
         \le Cr\int_{-r^2}^{0}\int_{Q'\cap\OO}\abs{\nabla u(X,t)}dXdt
         +|Q'\cap\OO|\int_{-r^2}^{0}|\avg{u}(t)-c_u|dt.
    \end{multline}
   Let $\eta=\eta(X)$ be a smooth function with $\supp\eta\subset 2Q'$, $\eta=1$ on $Q'$, $0\le \eta\le 1$, and $|\nabla\eta|\lesssim r^{-1}$. We define 
   \[
   \avg{u}_\eta(t):= \int_{\OO}u(Y,t)\eta(Y)dY \quad\text{for }t\in (T_0,T_1),
   \]
   and set $c_\eta:=\int_{\OO}\eta\,dX$. For $t,t'\in [-r^2,0]$, we have
   \begin{multline*}
       \avg{u}_\eta(t)-\avg{u}_\eta(t')=\int_{\OO}(u(X,t)-u(X,t'))\eta(X)dX
       =\int_{\OO}\int_{t'}^t \dr_\tau u(X,\tau)\eta(X)d\tau dX\\
       =-\int_{\OO}\int_{t'}^tA(X,\tau)\nabla u(X,\tau)\cdot\nabla \eta (X) d\tau dX-\int_{\OO}\int_{t'}^t   F(X,\tau)\nabla\eta(X)d\tau dX
   \end{multline*}
    since $u$ is a weak solution to $Lu=-\divg  F$ with zero Neumann data on $\pom\cap Q_0$ (see \eqref{1.5} and the remark after it). Using the bound on $\nabla\eta$, we obtain that 
    \begin{equation}\label{eq.avgt-t'}
        |\avg{u}_\eta(t)-\avg{u}_\eta(t')|\le \frac{C}{r}\int_{\OO\cap 2Q'}\int_{-r^2}^{0}\br{|\nabla u| +|  F|}dtdX \quad\text{for }t,t'\in [-r^2,0].
    \end{equation}
    On the other hand, 
    \begin{multline*}
        \abs{\avg{u}(t)-\frac1{c_\eta}\avg{u}_\eta(t)}
        =\abs{\int_{\OO}\frac{\eta(Y)}{c_\eta}\br{u(Y,t)-\fint_{Q'\cap\OO}u(X,t)dX}dY}\\
        \le \frac1{c_\eta}\int_{\OO\cap 2Q'}\abs{u(Y,t)-\fint_{Q'\cap\OO}u(X,t)dX}dY
        \le \frac{Cr}{c_\eta}\int_{\OO\cap 2Q'}|\nabla u(Y,t)|dY\\
        \le Cr\fint_{\OO\cap 2Q'}|\nabla u(Y,t)|dY,
    \end{multline*}
    where we have used the Poincar\'e inequality in the spatial variables in the second-to-last inequality and $c_\eta\ge |\OO\cap Q'|\gtrsim |\OO\cap 2Q'|$ in the last inequality. By this and \eqref{eq.avgt-t'}, we get that 
    \[
        \int_{-r^2}^{0}|\avg{u}(t)-c_u|dt=\int_{-r^2}^{0}\abs{\fint_{-r^2}^{0}\br{\avg{u}(t)-\avg{u}(t')}dt'}dt
        \le C r\fint_{\OO\cap 2Q'}\int_{-r^2}^{0}\br{|\nabla u|+|  F|}dtdX.
    \]
    Combining this estimate with \eqref{eq.u-cu}, one obtains that
    \[
     \iint_{J_r\cap\om}|u(X,t)-c_u|dXdt\le C r\int_{\OO\cap 2Q'}\int_{-r^2}^{0}\br{|\nabla u|+|  F|}dtdX,
    \]
    as desired.
\end{proof}

In the following lemma we show relations between the Dirichlet and Poisson-Dirichlet problems. We remark that although it can be shown that the solvability of \Dq$^{L^*}$, \PDq$^{L^*}$ and \PRp$^{L}$ are actually {\it equivalent}, we will only derive the partial implications that are sufficient for our purpose. 
\begin{lemma}\label{lem.DqPDq}
    Let $p\in(1,\infty)$. Then
    \[  \text{\Dq$^{L^*}$} \implies \text{\PDq$^{L^*}$} \implies  \text{\PRp$^{L}$}.\]
\end{lemma}
\begin{proof}
The implication \Dq$^{L^*}$ $\implies$ \PDq$^{L^*}$ follows with a few modifications from the argument in \cite[Section 3]{Ulm}. 
    Let $F\in L_c^\infty(\om,\Rn)$, and let $v\in \dot E_0(\om)$ be the energy solution to $L^*v=\divg F$ in $\om$ with $v=0$ on $\pom$. 
    For $(P,t)\in \pom$, $q\ge 1$, define the discretized area functional 
    \[
    T_q(F)(P,t):=\sum_{I\in \W(P,t)}\ell(I)\br{\fiint_{I^+}|F|^{2q}}^{1/{2q}}
    \]
    as in \cite[Definition 3.6]{Ulm}, where $\W(P,t):=\set{\pom\cap Q_{2^j}(P,t): j\in\Z}$ is a collection of parabolic cubes centered at $(P,t)$ on the boundary, and for $I\in\W(P,t)$, $I^+$ is a parabolic Whitney cube in $\om$ above $I$, with $\ell(I^+)=\ell(I)$. From \cite[Lemma 3.7]{Ulm} and the proof of \cite[Lemma 3.10]{Ulm}, it follows that if \Dq$^{L^*}$ is solvable, then for $q>\frac{n}{2}+1$,
    \begin{equation}\label{eq.Nv_Ulm}
        \|N(v)\|_{L^{p'}(\pom)}\lesssim\|T_q(F)\|_{L^{p'}(\pom)}.
    \end{equation}
    We claim that we can obtain
    \begin{equation}\label{eq.Nv-A}
        \|\wt N(v)\|_{L^{p'}(\pom)}\lesssim\|T_1(F)\|_{L^{p'}(\pom)}\lesssim\|\wt{\A}(\delta F)\|_{L^{p'}(\pom)}.
    \end{equation}
    Note that by \eqref{eq.C=A}, this means that \PDq$^{L^*}$ is solvable. 
    
    To see the first inequality in \eqref{eq.Nv-A}, we note that the constraint on $q$ in \eqref{eq.Nv_Ulm} comes only from the estimate for the local part 
    \[\wt v(X,t):=\iint_{Q_{X,t}}\nabla_Y G(X,t,Y,s)\cdot F(Y,s)dYds\] in the proof of \cite[Lemma 3.7]{Ulm}, where $Q_{X,t}:=Q_\frac{\delta(X,t)}{2}(X,t)$. This is because an estimate on the $L^q$ norm of the gradient of the Green function is used there to deduce that $\wt v(X,t)\lesssim T_q(F)(P,t)$ for $(X,t)\in I^+$, $I\in\W(P,t)$. However, this constraint can be removed and $q$ can be set equal to 1 if one takes average in the local part, which allows us to apply energy estimates. To be more precise, for fixed $(X,t)\in I^+$, with $I\in\W(P,t)$, we consider 
    \[
    \vv(Z,\tau):=\iint_{Q_{X,t}}\nabla_Y G(Z,\tau,Y,s)\cdot F(Y,s)dYds, \quad (Z,\tau)\in\om.
    \]
    Then $\vv\in\dot E_0(\om)$ solves $L\vv=\divg\br{F\1_{Q_{X,t}}}$ in $\om$.  Hence,  by Sobolev embedding $\dot E_0(\om)\hookrightarrow L^{\frac{2(n+2)}{n}}(\om)$ (c.f. subsection \ref{ESPP}) and the energy estimate, one obtains that
   \begin{multline*}
       \br{\fiint_{Q_{X,t}/2}|\vv(Z,\tau)|^{\frac{2(n+2)}{n}}dZd\tau}^{\frac{n}{2(n+2)}}
       \lesssim \delta(X,t)^{-\frac{n}{2}}\br{\iint_{\om}|\vv|^{\frac{2(n+2)}{n}}dZd\tau}^{\frac{n}{2(n+2)}}\\
       \lesssim \delta(X,t)^{-\frac{n}{2}}
      \|\vv\|_{\dot \E_0(\Omega)}
       \lesssim \delta(X,t)^{-\frac{n}{2}}\br{\iint_{Q_{X,t}}|F|^{2}dZd\tau}^{1/{2}}\\\lesssim\delta(X,t)\br{\fiint_{Q_{X,t}}|F|^{2}dZd\tau}^{1/{2}}
       \le T_1(F)(P,t).
   \end{multline*}
This modification to the proof of \cite[Lemma 3.7]{Ulm} yields the first inequality in \eqref{eq.Nv-A}.

\medskip 
Next, we justify the second inequality in \eqref{eq.Nv-A}. We shall prove the pointwise bound
\begin{equation}\label{eq.T-A}
    T_1(F)(P,t)\lesssim \wt\A_{1,2}^{a,c}(\delta F)(P,t) \quad \text{for } (P,t)\in\pom
\end{equation}
for some $a>0$ and $c\in(0,1)$. Then the desired the estimate follows from \eqref{eq.Aac=A}. To see \eqref{eq.T-A}, we observe that $I^+\subset Q_{c\delta(Y,s)/2}(Y,s)$ for all $(Y,s)\in I^+$ for some $c\in(0,1)$ close to 1. This implies that for $(P,t)\in\pom$, 
\begin{multline*}
   T_1(F)(P,t)\lesssim\sum_{I\in\W(P,t)}\ell(I)\fiint_{I^+}\br{\fiint_{Q_{c\delta(Y,s)/2}(Y,s)}|F|^2dZd\tau}^{1/2}dYds\\
   \lesssim \sum_{I\in\W(P,t)}\iint_{I^+}\br{\fiint_{Q_{c\delta(Y,s)/2}(Y,s)}|\delta F|^2dZd\tau}^{1/2}\frac{dYds}{\delta(Y,s)^{n+2}} \le \wt A^{a,c}_{1,2}(\delta F)(P,t) 
\end{multline*}
for some $a>0$ large enough so that the cone $\Gamma_a(P,t)$ contains $\bigcup_{I\in\W(P,t)}I^+$. This completes the proof of \eqref{eq.T-A} and thus \eqref{eq.Nv-A} and the implication  \Dq$^{L^*}$ $\implies$ \PDq$^{L^*}$. 
\medskip

The implication \PDq$^{L^*}$ $\implies$ \PRp$^L$ follows from duality. To see this, let $h\in L_c^\infty(\om)$, $F\in L_c^\infty(\om,\Rn)$, and let $u\in\dot E_0(\om)$ be the energy solution to $Lu=h+\divg F$. By \eqref{eq.NpC}, 
\[
\|\wt N(\nabla u)\|_{L^p(\pom)}\lesssim \abs{\iint_{\om}\nabla u\cdot G\,dXdt},
\]
for some $G\in L_c^\infty(\om)$ satisfying $\|\wt\C(\delta G)\|_{L^{p'}(\pom)}\le 1$. Let $v\in \dot E_0(\om)$ be the energy solution to $L^*v=\divg G$. Then using the PDEs satisfied by $u$ and $v$ gives 
\begin{equation}\label{eq.Ngradu-h+F}
    \|\wt N(\nabla u)\|_{L^p(\pom)}\lesssim \abs{-\iint_\om hv \,dXdt+\iint_\om F\cdot\nabla v\, dXdt }.
\end{equation}
By \eqref{eq.CNdual}, we have 
\[
\abs{\iint_\om hv \,dXdt}\lesssim \|\wt N(v)\|_{L^{p'}}\|\wt\C(\delta h)\|_{L^p}
\lesssim \|\wt\C(\delta G)\|_{L^{p'}}\|\wt\C(\delta h)\|_{L^p}\le \|\wt\C(\delta h)\|_{L^p}
\]
by the assumption that \PDq$^{L^*}$ is solvable. Again by \eqref{eq.CNdual}, we have
\[
\abs{\iint_\om F\cdot\nabla v\, dXdt}\lesssim\|\wt N(\delta \nabla v)\|_{L^{p'}}\|\wt\C(F)\|_{L^{p}}.
\]
By Caccioppoli's inequality, for any $(X,t)\in\om$,
\[
\br{\fiint_{Q_{X,t}/2}|\delta\nabla v|^2dYds}^{1/2}\lesssim \br{\fiint_{\frac23 Q_{X,t}}|v|^2}^{1/2}+\br{\fiint_{\frac23 Q_{X,t}}|\delta G|^2}^{1/2}.
\]
This implies that for any $\xi\in\pom$,
\[
\wt N(\delta\nabla v)(\xi)\lesssim \wt N_{1,2}^{2/3}(v)(\xi)+\wt\C^{2/3}_{1,2}(\delta G)(\xi),
\]
which yields that 
\[
\|\wt N(\delta \nabla v)\|_{L^{p'}}\lesssim \|\wt N(v)\|_{L^{p'}}+\|\wt\C(\delta G)\|_{L^{p'}}\lesssim 1.
\]
Combining these estimates with \eqref{eq.Ngradu-h+F}, one obtains that 
\[
\|\wt N(\nabla u)\|_{L^{p'}}\lesssim \|\wt\C(\delta h)\|_{L^p}+\|\wt\C(F)\|_{L^{p}},
\]
as desired. 
\end{proof}

\section{Proof of Theorem~\ref{thm.NtoLoc}}\label{S4}
In this section, we prove Theorem~\ref{thm.NtoLoc}. We begin with  the following lemma.
\begin{lemma}\label{lem.N+D=wPNR}
Let $p\in(1,\infty)$. Then 
    \[ \text{\Np$^L$ $+$ \Dq$^{L^*}$ $\implies$ \wPRNp$^L$ %$\implies$ \Np$_L$.}
    }\]
\end{lemma}
\begin{proof}
    Let $F \in {L^\infty_c}(\Omega)$, and let $u\in \dot E(\om)$ be the energy solution to $Lu=-\divg F$ in $\Omega$ with zero Neumann data. We want to prove that 
\[\|\wt N(\nabla u)\|_{L^p(\partial \Omega)} \leq C \|\wt \C(F)\|_{L^p(\partial \Omega)}\]
with a constant $C>0$ independent of $F$.

Let $u_D \in \dot\E_0(\Omega)$ be the energy solution of $Lu_D=-\divg F$ in $\Omega$ with zero Dirichlet data.
Set $w:=u-u_D$. We claim that $w\in \dot E(\om)$ is the energy solution to  $Lw=0$ in $\Omega$ with $\dr_\nu^A w|_{\pom}=g$ for some $g\in L^p(\pom)\cap \dot H^{-1/4}_{\dr_t-\Delta_x}(\pom)$ that satisfies 
\begin{equation}\label{eq.gLp}
    \|g\|_{L^p(\pom)}\le \|\wt N(\nabla u_D)\|_{L^p(\pom)}.
\end{equation}
Recall that $\dr_\nu^A w|_{\pom}=g$ is in the weak sense (as in \eqref{eq.NeumannBdy}). Once \eqref{eq.gLp} is established, by the assumption that \Np$^{L}$ and \Dq$^{L^*}$ hold and Lemma~\ref{lem.DqPDq}, we obtain that 
\begin{multline} \label{Np1}
\|\wt N(\nabla u)\|_{L^p(\partial \Omega)}\le \|\wt N(\nabla u_D)\|_{L^p(\partial \Omega)}+
\|\wt N(\nabla w)\|_{L^p(\partial \Omega)}\\ \le \|\wt N(\nabla u_D)\|_{L^p(\partial \Omega)}+\|g\|_{L^p(\partial \Omega)} \le 2\|\wt N(\nabla u_D)\|_{L^p(\partial \Omega)} \lesssim \|\wt \C(F)\|_{L^p(\partial \Omega)}. 
\end{multline}
\smallskip

It remains to prove that there exists some $g\in L^p(\pom)\cap \dot H^{-1/4}_{\dr_t-\Delta_x}(\pom)$ such that 
\begin{equation}\label{eq.drnw=g}
    \iint_{\mathcal O\times\R}\left[A \nabla w \cdot{\nabla \varphi} + \HT \dhalf w \cdot {\dhalf \varphi}\right] \d X \d t=\int_{\dr\OO\times\R}g\vp \quad \text{for all }\vp\in\dot E,
\end{equation}
and that $g$ satisfies the bound \eqref{eq.gLp}.

For simplicity we assume our domain is 
$\Omega=\R^n_+\times\R$, the generalization to any bounded or unbounded Lipschitz domain follows by localization and partition of unity.

Note that since $u\in \dot E(\om)$ solves $Lu=-\div F$ in $\om$ and $u$ has zero Neumann data, the left-hand side of \eqref{eq.drnw=g} is equal to 
\[
 -\iint_{\mathcal O\times\R}\left[A \nabla u_D \cdot{\nabla \varphi} + \HT \dhalf u_D \cdot {\dhalf \varphi}-F\cdot\nabla\vp\right] \d X \d t.
\]
%We start with boundedness of the $\|\wt N(\nabla u_D)\|_{L^p(\partial \Omega)}<\infty$ by Theorem \ref{thm.D-PR}
%and the fact that because $F$ is bounded and compactly supported we must also have $u_D\in\dot \E_0(\Omega)$, i.e., $u_D$ is an energy solution. Thus $\nabla u_D,D^{1/2}_tu_D\in L^2(\Omega)$.
%This allows us to consider for all $\varphi\in \dot \E(\Omega)$ and some small fixed $\delta>0$ the linear functional
Therefore, we consider the linear functional $\Lambda$ defined for $\vp\in \dot H^{1/4}_{\dr_t-\Delta_x}(\pom)$ by
\begin{equation}\label{eqdef.Lbd}
    \Lambda(\varphi):=
%a(u_D,\varphi):=
-\iint_{\mathcal O\times\R} \left[A \nabla u_D \cdot{\nabla \Phi} + \HT \dhalf u_D \cdot {\dhalf \Phi}-F\nabla\Phi\right] \d X \d t,
\end{equation}
where $\Phi$ is any extension of $\vp$ in $\dot E(\om)$. Note that this is well-defined because if $\Phi_1$ and $\Phi_2$ are two extensions of $\vp$ in $\dot E(\om)$, then $\Phi_1-\Phi_2\in \dot E_0(\om)$ and so 
\[
\iint_{\mathcal O\times\R} \left[A \nabla u_D \cdot{\nabla (\Phi_1-\Phi_2)} + \HT \dhalf u_D \cdot {\dhalf (\Phi_1-\Phi_2)}-F\nabla(\Phi_1-\Phi_2)\right] \d X \d t =0 
\]
since $L u_D=-\divg F$. Moreover, $\Lambda$ is bounded because for any $\phi\in \dot H^{1/4}_{\dr_t-\Delta_x}(\pom)$, we have that 
\[
|\Lambda(\vp)|\lesssim \br{\|u_D\|_{\dot E}+\|F\|_{L^2(\om)}}\|\Phi\|_{\dot E(\om)}\lesssim\br{\|u_D\|_{\dot E}+\|F\|_{L^2(\om)}}\|\vp\|_{\dot H^{1/4}_{\dr_t-\Delta_x}(\pom)},
\]
where we have used that $\|F\|_{L^2(\om)}<\infty$ since $F\in L_c^\infty(\om)$. Hence, there exists a unique $g\in\Hdot^{-1/4}_{\pd_{t} - \Delta_x}(\partial \Omega)$ such that
$\Lambda(\varphi)=\langle g,\varphi\rangle$ for all  $\vp\in \dot H^{1/4}_{\dr_t-\Delta_x}(\pom)$. As Lipschitz functions 
$\vp:\R^{n+1}\to\R$ with bounded support restricted to $\pom$ are dense in 
$ \dot H^{1/4}_{\dr_t-\Delta_x}(\pom)$ it follows that the identity 
$\Lambda(\varphi)=\langle g,\varphi\rangle$ holds for all such $\varphi$ (and we can take $\Phi$ to be their Lipschitz extension in \eqref{eqdef.Lbd}) and still uniquely determined $g$ in the space $\Hdot^{-1/4}_{\pd_{t} - \Delta_x}(\partial \Omega)$.

We argue that $g\in L^p(\pom)$ with the bound \eqref{eq.gLp} by showing that the co-normal averages defined for any $r>0$ by
\begin{equation}
g_r(x,t)=-\fiint_{B_{r/2}(x,r,t)} A(y,y_n,s)\nabla u_D(y,y_n,s)\cdot e_n dY\,ds
\end{equation}
converge weakly to $g$ in $L^p(\partial\Omega)$ as $r\to 0$. The bound \eqref{eq.gLp} would then follow from the observation that 
$$\|g_r\|_{L^p(\pom)}\le \|\wt N(\nabla u_D)\|_{L^p(\partial \Omega)}<\infty,\quad\mbox{ for all $r>0$}.$$

To show the weak convergence of $g_r$ to $g$ in $L^p(\pom)$, it suffices to show that
for any Lipschitz function $\varphi:\R^{n+1}\to\R$ with bounded support, we have that
\begin{equation}\label{eq-sesqx}
 \iint_{ \Omega} \left[A \nabla u_D \cdot {\nabla \varphi} -u_D \partial_t\varphi-F\nabla \varphi\right] d X d t= \lim_{r\to0_+}\int_{\partial\Omega}g_r(x,t)\varphi(x,r,t)dx\,dt.
 \end{equation}
%Hence, $(A\nabla u_D)_n$ is understood to be equal to $g$.
To see this, we fix $r>0$ and consider the sets 
$$R_r=\bigcup_{(x,t)\in\partial\Omega}B_{r/2}(x,r,t),\qquad S_{X,t}=\{(y,s)\in\partial\Omega: (X,t)\in B_{r/2}(y,r,s)\},$$
 for any $(X,t)\in\Omega$. Finally, let $\sigma(X,t)=\sigma(S_{(X,t)})$. Clearly $\sigma(X,t)>0$ if and only if 
 $(X,t)\in R_r$ and moreover $\sigma(X,t)$ is a function that only depends on $x_n$ (the distance of the point
 $(X,t)$ to the boundary) and hence we shall also write $\sigma(x_n)$. Observe that
 $$\int_0^\infty\sigma(x_n)dx_n=c_nr^{n+2},\quad\mbox{and hence }\quad 1=\frac1{c_nr^{n+2}}\int_0^\infty\sigma(x_n)dx_n.$$
 Hence for any point $(X,t)\in \Omega$ we define $\theta_r(X,t)=\theta_r(x,x_n,t)$ as
$$ \theta_r(X,t)=\frac1{c_nr^{n+2}}\int_0^{x_n}\sigma(s)ds.$$
 By what we have observed above, $\theta_r$ vanishes near the boundary $\partial\Omega$ and equal to $1$ when $x_n>3r/2$. It follows that since $Lu_D=-\divg F$, we have for the test function $\varphi\theta_r$ that 
 \begin{equation}\label{eq.testvt}
     \iint_{ \Omega} \left[A \nabla u_D \cdot {\nabla (\varphi\theta_r)} -u_D\partial_t(\varphi \theta_r)-F\nabla (\varphi\theta_r)\right] d X d t=0.
 \end{equation}
  It follows, since the support of $F$ is away from the boundary, that $\iint_\Omega F\nabla (\varphi\theta_r)=
 \iint_\Omega F\nabla \varphi$ for all $r$ sufficiently small. For the second term on the left-hand side of \eqref{eq.testvt}, as $\theta_r$ does not depend on $t$ we have 
$$\iint_\Omega u_D\partial_t(\varphi \theta_r)=\iint_\Omega \theta_ru_D\partial_t \varphi\to \iint_\Omega u_D\partial_t \varphi,\quad\mbox{as }r\to 0_+.$$
 The first term can be written as sum of two terms (depending on whether the gradient hits $\varphi$ or $\theta$), i.e.,
 \begin{equation}\label{eq232}
 \iint_{ \Omega} A \nabla u_D \cdot {\nabla (\varphi\theta_r)} = 
  \iint_{ \Omega} A \nabla u_D \cdot ({\nabla \varphi)\theta_r} + \iint_{ \Omega} A \nabla u_D \cdot ({\nabla \theta_r)\varphi}
 \end{equation}
 and clearly
 $$  \iint_{ \Omega} A \nabla u_D \cdot ({\nabla \varphi)\theta_r} 
   \to  \iint_{ \Omega} A \nabla u_D\cdot \nabla \varphi\quad\mbox{as }r\to0_+.
 $$ 
 It remains to consider the second term of \eqref{eq232}. Here $\theta_r$ is only a function of $x_n$ and hence
 \begin{equation}\label{eq233}
 \iint_{ \Omega} A \nabla u_D \cdot ({\nabla \theta_r)\varphi}=\iint_{R_r} A_{n_j}\partial_j u_D (\partial_n \theta_r)\varphi=c_n^{-1}r^{-n-2}\iint_{R_r} A_{n_j}\partial_j u_D\sigma(X,t)\varphi(X,t),
 \end{equation} 
since
$$\partial_n\theta_r(X,t)=\frac1{c_nr^{n+2}} \sigma(X,t).$$
However, by Fubini's theorem  
\begin{multline}\label{eq234} 
\int_{\partial\Omega}g_r(x,t)\varphi(x,r,t)dxdt=\\-c_n^{-1}r^{-n-2}\iint_{R_r}A(X,t)\nabla u_D(X,t)\cdot e_n
\sigma(X,t)\left(\fiint_{S_{X,t}}\varphi(Y,s)dY\,ds\right)dx\,dt,
\end{multline}  
where $\fiint_{S_{X,t}}\varphi(Y,s)dY\,ds\to \varphi(X,t)$ as $r\to0_+$ by the Lipschitz character of $\varphi$.
It follows that using \eqref{eq233}
$$\iint_{ \Omega} A \nabla u_D \cdot ({\nabla \theta_r)\varphi}+\int_{\partial\Omega}g_r(x,t)\varphi(x,r,t)\to 0,\qquad\mbox{as }r\to0_+.$$
After putting everything together the claim \eqref{eq-sesqx} follows.
% \medskip

% Next, since $w$ solves the homogeneous Neumann problem with datum $-g$, which thanks to the fact that
% $g\in L^p(\partial\Omega)\cap \Hdot^{-1/4}_{\pd_{t} - \Delta_x}(\partial \Omega)$ is also an energy solution and therefore 
% \begin{equation}\nonumber
%  \iint_{ \Omega} \left[A \nabla w \cdot {\nabla \varphi} -w \partial_t\varphi\right] d X d t= -\int_{\partial\Omega}g(x,t)\varphi(x,0,t)dx\,dt,
%  \end{equation}
% for all Lipschitz function $\varphi:\R^{n+1}\times\R$ with bounded support. This combined with \eqref{eq-sesqx} yields for $u=u_D+w$ 
% \begin{equation}\nonumber
%  \iint_{ \Omega} \left[A \nabla u \cdot {\nabla \varphi} -u \partial_t\varphi-F\nabla\varphi\right] d X d t= 0,
%  \end{equation}
% for all such Lipschitz function $\varphi$. By density this them implies the identity
% \begin{equation}\label{eq-sesqxx}
%  \iint_{ \Omega} \left[A \nabla u \cdot {\nabla v} +H_tD^{1/2}_tu \cdot D^{1/2}_tv-F\nabla v\right] d X d t= 0,
%  \end{equation}
% $\forall v\in\dot \E(\Omega)$. Thus, indeed $u=u_D+w$ is an energy solution that solves the weak Neumann Poisson problem as required. 
\end{proof}

We now show the duality between the weak Poisson-Neumann problem and the weak Poisson-Neumann-Regularity problem. Although one can deduce one from another, we only show one direction as it is sufficient for the purpose of this paper.
\begin{lemma}\label{lem.wPNR=wPN}
    Let $p\in (1,\infty)$. Then 
    \[
    \text{\wPRNp } \implies \text{\wPNq}^{L^*}.
    \]
\end{lemma}
\begin{proof}
    Let $  F \in L^\infty_c(\Omega)$ and let $u$ be the energy solution to $L^*u= - \div   F$ in $\om$ with zero Neumann data. We want to prove that 
\[\|\wt N(\delta\nabla u)\|_{L^{p'}(\partial \Omega)} \lesssim \|\wt \C(\delta |  F|)\|_{L^{p'}(\partial \Omega)}.\]

Let $K$ be any compact subset of $\Omega$. In this case $\|\wt N(\1_K \delta\nabla u)\|_{p'} < +\infty$, and so by the duality \eqref{eq.NpC}, and the density of $L^\infty_c(\Omega)$ in the space $\wt T^p_1(\Omega)$, there exists $  G =   G_K \in L^\infty_c(\Omega)$ such that 
\begin{equation} \label{eq.CdG}
    \|\wt \C(\delta   G)\|_{L^{p}(\partial \Omega)} \lesssim 1
\end{equation}
and
\[I := \|\wt N(\1_K \delta\nabla u)\|_{L^{p'}(\partial \Omega)} \leq \iint_\Omega \delta\nabla u\cdot  G\, dXdt=-\iint_{\om}u\divg (\delta  G) dXdt.\]

We let $v\in \dot E(\om)$ be the energy solution to $L v= -\divg(\delta   G)$ in $\Omega$ with zero Neumann data, and the bound on $I$ becomes 
%({\color{red}here we use $u$ as a test function in the equation satisfied by $v$, and use $v$ as a test function in the equation satisfied by $u$.})
\begin{multline}\label{eq.wPNuLv}
    I \leq \iint_\Omega u\,  Lv \, dX = 
    \iint_\om u\br{-\dr_t v+\divg A\nabla v}dXdt
    =\iint_\Omega \br{\dr_t u\, v - A^T \nabla u\cdot \nabla v }\, dXdt\\
    = \iint_\Omega   F \cdot \nabla v \, dXdt
\end{multline} 
because $u$ is a weak solution to $L^* u = -\div   F$ with zero Neumann data.
By the Carleson inequality \eqref{eq.CNdual}, the assumption that \wPRNp\ holds for $v$, and \eqref{eq.CdG}, we have that
\begin{multline*} 
\|\wt N(\1_K\delta \nabla u)\|_{L^{p'}(\partial \Omega)} = I \lesssim \|\wt \C(\delta   F)\|_{L^{p'}(\partial \Omega)} \|\wt N(\nabla v)\|_{L^{p}(\pom)} \\
\lesssim \|\wt \C(\delta   F)\|_{L^{p'}(\partial \Omega)} 
\|\wt \C(\delta   G)\|_{L^{p}(\partial \Omega)} \lesssim \|\wt \C(\delta   F)\|_{L^{p'}(\partial \Omega)}.
\end{multline*}
Since the bound is independent of the compact $K$, we take $K \uparrow \Omega$ and we have the desired result.
\end{proof}

\begin{lemma}\label{LemmaM}
    Let $p\in(1,\infty)$. Suppose that \wPNq$^{L^*}$ is solvable in $\om=\mathcal O\times\R$.  Let $Q_0:=Q_0'\times(T_0,T_1)$ be a cube centered at $\pom$. Then for any $(\bar x,\bar t)\in Q_0\cap\pom$, any backward parabolic cube $J_r(\bar x,\bar t):=Q_r'(\bar x)\times(\bar t-r^2,\bar t)$ that satisfies $J_{2r}(\bar x,\bar t)\subset Q_0$, and any weak solution $u$ to $Lu=0$ in $Q_0\cap\om$ satisfying $u \in L^2\left((T_0,T_1); W^{1,2}(Q_0'\cap\OO)\right)$, $\dr_t u\in L^2\left((T_0,T_1); W^{-1,2}(Q_0'\cap\OO)\right)$ and having zero Neumann data on $Q_0\cap\pom$, we have
\begin{equation} \label{loc1}
\|\wt N(|\nabla u|\1_{J_r(\bar x,\bar t)})\|_{L^p(\partial \Omega)} \leq C r^{(n+1)/p}\left ( \fiint_{J_{2r}(\bar x,\bar t)\cap \om} |\nabla u|^2 dXdt \right)^\frac12.
\end{equation}
\end{lemma}

\begin{proof}
The proof follows the strategy of \cite[Lemma 4.1]{FL}, adapted to the parabolic setting. Fix $J_r:=J_r(\bar x,\bar t)$ with $(\bar x,\bar t)\in Q_0\cap\pom$  and $J_{2r}:=J_{2r}(\bar x,\bar t)\subset Q_0$. We claim that we can assume that for some $K\geq 1$ (independent of $J_r$ and $u$) to be determined during the proof, $J_{Kr}(\bar x,\bar t)\subset Q_0$, and that the desired estimate \eqref{loc1} would follow once we prove that 
\begin{equation} \label{loc2}
\|\wt N(|\nabla u|\1_{J_r})\|_{L^p(\partial \Omega)} \lesssim r^{(n+1)/p}\Big( \fiint_{J_{Kr} \cap \Omega} |\nabla u|^2 dXdt \Big)^\frac12.
\end{equation} 
In fact, if \eqref{loc2} holds, then 
%for a given backward parabolic cube $J_r(\bar x,\bar t)$ with $(\bar x,\bar t)\in Q_0\cap\pom$ and $J_{2r}(\bar x,\bar t)\subset Q_0$, 
we cover $J_r$ by a finite collection of backward parabolic cubes $J_{i}:=J_{r_i}(x_i,t_i)$ with $r_i\approx r$ such that 
either 

(i) $(x_i,t_i)\in\pom$ and $J_{Kr_i}(x_i,t_i)\subset J_{2r}$, or 

(ii) $J_{2r_i}(x_i,t_i)\subset J_{2r}\cap\om$.

As there are finitely many $J_i$ and the number depends only on the dimension, we have that 
\[
\|\wt N(|\nabla u|\1_{J_r})\|_{L^p(\partial \Omega)} \lesssim \sup_{i} \|\wt N(|\nabla u|\1_{J_i})\|_{L^p(\partial \Omega)}.
\]
If $J_i$ is as in (i), then we use \eqref{loc2} to get that 
\[
\|\wt N(|\nabla u|\1_{J_i})\|_p\lesssim r_i^{(n+1)/p}\br{\fiint_{J_{Kr_i}(x_i,r_i) \cap\om}\abs{\nabla u}^2}^{1/2}\lesssim r^{(n+1)/p}\br{\fiint_{J_{2r}\cap\om}\abs{\nabla u}^2}^{1/2}.
\]
If $J_i$ is as in (ii), then we observe that $\wt N(|\nabla u|\1_{J_i})$ is supported on a large shadow of $J_i$ on $\pom$, which is of measure $Cr^{n+1}$. Therefore,
\[ \|\wt N(|\nabla u|\1_{J_i})\|_{L^p(\partial \Omega)} \approx r^{(n+1)/p} \left(\fiint_{J_i} |\nabla u|^2 \, dX \right)^\frac12 \approx r^{(n+1)/p} \left(\fiint_{J_{2r}\cap\om} |\nabla u|^2 \right)^\frac12.\]
Combining both cases, we have shown \[
\|\wt N(|\nabla u|\1_{J_r})\|_{L^p(\partial \Omega)} \leq C r^{(n+1)/p}\left ( \fiint_{J_{2r}\cap \om} |\nabla u|^2 dXdt \right)^\frac12,
\] as desired. 
\smallskip

 We choose a piecewise smooth function $\vp$ so that $\varphi \equiv 1$ on $J_{\frac32r}$, $\varphi \equiv 0$ outside $J_{2r}$, $0\le \vp\le 1$,  $|\nabla \varphi| \lesssim \frac1r$ and that $|\dr_t\varphi|\lesssim\frac{1}{r^2}$ in $J_{2r}$. 
By duality, i.e. \eqref{eq.NpC}, there exists $  G \in L^\infty_c(\Omega,\Rn)$ such that
\begin{equation} \label{locG}
    \|\wt \C(\delta   G)\|_{L^{p'}(\partial \Omega)} \leq 1
\end{equation}
and
\begin{equation} \label{loc3}
    \|\wt N(\nabla u \1_{J_r})\|_{L^p(\partial \Omega)} \leq C \iint_\Omega \1_{J_r}\nabla u \cdot   G\, dXdt
\end{equation}
%Without loss of generality, we can assume that $  G =   G \1_Q$. 
Let $c_u$ be a constant to be determined later. We have that
\begin{multline*}
    \|\wt N(\nabla u \1_{J_r})\|_{L^p(\partial \Omega)} \lesssim \iint_\Omega  \1_{J_r}\nabla u \cdot   G\, \varphi \, dXdt 
    =\iint_\Omega  \1_{J_r}\nabla(u-c_u) \cdot   G\, \varphi \, dXdt\\
    = \int_{T_0}^{\bar t}\int_{\OO}\1_{J_r} \nabla\br{(u-c_u) \varphi }\cdot  G \, dXdt- \iint_{J_r\cap\om} (u-c_u) \nabla \varphi \cdot   G \, dXdt\\
    =\int_{T_0}^{\bar t}\int_{\OO}\1_{J_r} \nabla\br{(u-c_u) \varphi }\cdot  G \, dXdt=: I_0,
\end{multline*}
where we have used the observation that $ \iint_{J_r\cap\om} (u-c_u) \nabla \varphi \cdot   G \, dXdt=0$ since $\nabla\vp =0$ in $J_{\frac32r}\supset J_{r}$.

% The term $I_1$ is easy to treat. We choose $c_u=\fiint_{J_{8r}\cap \om}u\,dXdt$. By the bound $|\nabla\vp|\lesssim r^{-1}$, we have that  
% \[
% |I_1|\lesssim r^{n+1}\br{\fiint_{J_r\cap\om}|  G(X,t)|dXdt}\sup_{J_r\cap\om}|u-c_u|\lesssim r^{n+1}\br{\inf_{(x,t)\in J_r\cap\pom}\wt C(\delta G)(x,t)}\sup_{J_r\cap\om}|u-c_u|
% \]
% by the definition of $\wt\C$. Since $u-c_u$ is also a solution to $Lu=0$ in $Q_0\cap\om$ with zero Neumann data on $Q_0\cap \pom$, we have by Lemma~\ref{lem.MoserNeu} and  Lemma~\ref{lem.Pnc_sol} (recall $J_{Kr}\subset Q_0$) that 
% \begin{equation}\label{eq.supu-cu.1}
%     \sup_{J_r\cap\om}|u-c_u|\le C\fiint_{J_{4r}\cap\om}|u-c_u|dXdt\le C\,r\fiint_{J_{8r}\cap\om}|\nabla u|dXdt.
% \end{equation}
% Since
% \[\inf_{x\in J_r \cap \partial \Omega} \wt \C_1(\delta G)(x) \lesssim r^{(-n-1)/p'}\|\wt \C_1(\delta G) \|_{L^{p'}(\partial \Omega)} \lesssim r^{(-n-1)/p'} \]
% by \eqref{locG}, and that $1-\frac1{p'} = \frac1p$, we deduce that
% \[I_1 \lesssim r^{(n+1)/p} \left( \fiint_{J_{8r} \cap \Omega} |\nabla u|^2 \right)^\frac12\]
% as desired.
To treat the term $I_0$, we define $v\in \dot E$ to be the energy solution to $L^*v= - \divg   (G\1_{J_r})$ in $\Omega$ with zero Neumann data. We take $(u-c_u)\vp$ as a test function in the PDE satisfied by $v$---which is valid since  $(u-c_u)\vp\in L^2(\R,\W^{1,2}(Q_0'\cap\OO))$; see also \eqref{eq.wksol_cutoff} and the remark below---and integrate in time from $T_0$ to $\bar t$ to get that\footnote{We shall be somewhat imprecise regarding the domain of integration in the spatial variables: we write the integral over $\OO$, although $u$ is only defined on $Q_0 \cap \OO$. This causes no problem, since we may extend $u$ arbitrarily outside $Q_0$, as the cutoff function $\varphi(\cdot, t)$ is compactly supported in $J_{2r} \cap \OO\subset Q_0\cap\OO$.
}  
\begin{multline*}
I_0 
= \int_{T_0}^{\bar t} \langle\dr_t v(\cdot,t), (u(\cdot,t)-c_u)\vp(\cdot,t)\rangle_{W^{-1,2},\W^{1,2}}dt
 -\int_{T_0}^{\bar t}\int_{\OO} A^T\nabla v\cdot\nabla[(u-c_u)\varphi]dXdt\\ 
=  \int_{T_0}^{\bar t}\langle\dr_t v(\cdot,t), (u(\cdot,t)-c_u)\vp(\cdot,t)\rangle_{W^{-1,2},\W^{1,2}}dt
-\int_{T_0}^{\bar t}\int_{\OO} A \nabla [(u-c_u)\varphi] \cdot \nabla v \, dXdt.
\end{multline*}
We want to use the PDE satisfied by $u-c_u$, and so we further rewrite the second term using the product rule:
% \begin{multline*}
%     -\int_{T_0}^{\bar t}\int_{\OO} A \nabla [(u-c_u)\varphi] \cdot \nabla (v-c_v) \, dXdt
%     =-\int_{T_0}^{\bar t}\int_{\OO} A\nabla (u-c_u)\cdot\nabla(\vp (v-c_v))dXdt\\ 
% +\int_{T_0}^{\bar t}\int_{\OO} A\nabla (u-c_u)\cdot\nabla\vp\, (v-c_v)\,dXdt
% -\int_{T_0}^{\bar t}\int_{\OO} A\nabla\vp\cdot\nabla v\,(u-c_u)dXdt.
% \end{multline*}
\begin{multline*}
    -\int_{T_0}^{\bar t}\int_{\OO} A \nabla [(u-c_u)\varphi] \cdot \nabla v \, dXdt
    =-\int_{T_0}^{\bar t}\int_{\OO} A\nabla (u-c_u)\cdot\nabla(\vp\, v)dXdt\\ 
+\int_{T_0}^{\bar t}\int_{\OO} A\nabla (u-c_u)\cdot\nabla\vp\,v\,dXdt
-\int_{T_0}^{\bar t}\int_{\OO} A\nabla\vp\cdot\nabla v\,(u-c_u)dXdt.
\end{multline*}
For the first term coming from $I_0$, we claim that 
\begin{multline}\label{eq.pairID}
     \int_{T_0}^{\bar t} \langle\dr_t v(\cdot,t), (u(\cdot,t)-c_u)\vp(\cdot,t)\rangle_{\W^{-1,2},\W^{1,2}}dt\\
 =-\int_{T_0}^{\bar t} \langle\dr_t(u-c_u), v(\cdot,t)\vp(\cdot,t)\rangle_{\W^{-1,2},\W^{1,2}}dt
 -\int_{T_0}^{\bar t}\int_{\OO} (u-c_u)v\,\dr_t\vp\,dXdt.
\end{multline}
%(u-c_u)\dr_t(\vp (v-c_v))dXdt-\iint_{\om}A\nabla (u-c_u)\cdot\nabla(\vp (v-c_v))dXdt \\
%-\iint_{\om}(u-c_u) (v-c_v)\,\dr_t\vp\,dXdt
%+\iint_\om A\nabla u\cdot\nabla\vp\, (v-c_v)\,dXdt\\
%-\iint_\om A\nabla\vp\cdot\nabla v(u-c_u)dXdt.
In fact, by a density argument, one can assume that  $u,v\in C^\infty(\OO\times (T_0,T_1))$ and then the pairing of $\W^{-1,2}(\OO)$ and $\W^{1,2}(\OO)$ becomes inner product in $L^2(\OO)$. In addition, observe that since $v$ is a solution to $L^*v=0$ in $\OO\times [\bar t,\infty)$ with 0 Neumann data, $v$ is constant in $\set{t\ge \bar t}$ (notice that we use here that $L^*$ is backward in time parabolic PDE). Without loss of generality, we can assume that the constant is 0, that is, 
\begin{equation}\label{eq.v=0}
    v(X,t)=0 \qquad\text{for $t\ge \bar t$}.
\end{equation}
Therefore, integrating by parts in $t$, the left-hand side of \eqref{eq.pairID} can be written as 
\begin{multline*}
     \int_{T_0}^{\bar t}\int_{\OO}\dr_tv(u-c_u)\vp\,dXdt\\
     =\int_{\OO}\int_{T_0}^{\bar t}\dr_t\br{v(u-c_u)\vp}dtdX
     -\int_{\OO}\int_{T_0}^{\bar t}\dr_t(u-c_u)v\,\vp\,dtdX
     -\int_{T_0}^{\bar t}\int_{\OO} (u-c_u)v\,\dr_t\vp\,dXdt\\
     =-\int_{T_0}^{\bar t}\int_{\OO}\dr_t(u-c_u)v\,\vp\,dtdX
     -\int_{T_0}^{\bar t}\int_{\OO} (u-c_u)v\,\dr_t\vp\,dXdt
\end{multline*}
thanks to \eqref{eq.v=0}.

Since $u-c_u$ satisfies $Lu = 0$ in $Q_0\cap \Omega$ with zero Neumann data on $\pom\cap Q_0$ (recall that $J_{Kr}\subset Q_0$), 
we take  $\vp\,v$ as a test function and integrate in time from $T_0$ to $\bar t$ to get that
\begin{equation*}
   \int_{T_0}^{\bar t} \langle\dr_t(u-c_u), v(\cdot,t)\vp(\cdot,t)) \rangle_{\W^{-1,2},\W^{1,2}}dt+ \int_{T_0}^{\bar t} \int_{\OO}A\nabla (u-c_u)\cdot\nabla(\vp \,v)dXdt =0.
\end{equation*}
Therefore, 
\begin{multline}\label{eq.I0split}
I_0
=  -\int_{T_0}^{\bar t}\int_{\OO} (u-c_u)v\,\dr_t\vp\,dXdt
+\int_{T_0}^{\bar t}\int_{\OO} A\nabla u\cdot\nabla\vp\, v\,dXdt\\
-\int_{T_0}^{\bar t}\int_{\OO}  A\nabla\vp\cdot\nabla v(u-c_u)dXdt
=:I_{00}+I_{01}+I_{02}.
%=:I_{01}+I_{02}.
\end{multline}

To estimate these terms, we first derive the following estimate: let $Q:=Q_{r}(\bar x,\bar t)$ and $kQ:=Q_{kr}(\bar x,\bar t)$ for $k>0$, 
\begin{equation}\label{eq.v-Nv}
    \fiint_{4Q\cap\om}|v|dXdt\lesssim r^{-n-1}
    \br{\|\wt N(\delta|\nabla v|\1_{4KQ})\|_{L^1(\pom)}+\|\wt\C(\delta   G\1_{4KQ})\|_{L^1(\pom)}}.
\end{equation}
%with a constant $c_v$ to be specified shortly.

To prove \eqref{eq.v-Nv}, let $\mathcal W$ be a Whitney decomposition of $\om$ into parabolic cubes. Then 
\begin{multline}\label{eq.v-cv.sum}
\iint_{4Q\cap\om}|v|dXdt
\lesssim\int_{(y,s)\in\pom}\iint_{\Gamma(y,s)}|v(X,t)|\1_{4Q}(X,t)\frac{dXdt}{\delta(X,t)^{n+1}}d\sigma(y,s)\\
     \lesssim \int_{(y,s)\in\pom}\sum_{W\in\mathcal W,\, (y,s)\in 10 W}\ell(W)\fiint_{W}|v(X,t)|\1_{4Q}(X,t)dXdt\,d\sigma(y,s).
\end{multline}
Let $W_Q\in\mathcal W$ be a parabolic Whitney cube associated to $Q$, that is, $W_Q\subset Q$, $10W_Q\cap\pom\neq\emptyset$, and $\delta(W_Q,\pom):=\dist_p(W_Q,\pom)\gtrsim r$. Moreover, we choose $W_Q\subset \set{t>\bar t}$, and so $\fiint_{W_Q}v\,dXdt=0$ in view of \eqref{eq.v=0}.
%We take $c_v=\fiint_{W_Q}v\,dXdt$. 

Fix any $W\in\mathcal W$ such that $W\cap 4Q \neq \emptyset$. We construct a Harnack chain of Whitney cubes $W=W_1,\dots,W_{M}=W_Q$. Such a Harnack chain satisfies 
\begin{enumerate}
    \item $M \lesssim 1 + \ln(r/\ell(W))$;
    \item for any $1\leq i \leq M$ we have $2W_i \in \gamma^*(y,s)$ for any $(y,s)\in 10W \cap \partial \Omega$ and any parabolic cone $\gamma^*(y,s)$ with large enough aperture;
    \item $W_i \subset 4KQ$ for $K$ large enough (as $W \cap 4Q \neq \emptyset$). 
\end{enumerate}
%For the fixed $W$, suppose that it is centered at $(Z,\tau)=(z,z_n,\tau)$. We write $W=W(z,z_n,\tau)=W'_Z\times I_{z_n}(\tau)$, where $W'_Z:=\set{X=(x_1,\dots,x_n)\in\Rn: |x_i-z_i|<\frac{z_n}{4},\, 1\le i\le n }$ is a Whitney cube in $\Rn$ centered at $Z$, and $I_{z_n}(\tau):=[\tau-\frac{z_n^2}{16},\tau+\frac{z_n^2}{16}]$.

By the triangle inequality, we have (recall that $0=\fiint_{W_Q}v\,dXdt=(v)_{W_M}$) 
\begin{equation}\label{eq.v-cv.sp}
   \fiint_{W}|v(X,t)|\1_{4Q}(X,t)dXdt\le \fiint_W|v(X,t)-(v)_W|\1_{4Q}dXdt+\sum_{i=2}^M|(v)_{W_{i-1}}-(v)_{W_i}|, 
\end{equation}
where we have used the notation $(v)_Q:=\fiint_Q v\,dXdt$. By Lemma~\ref{lem.Pnc_sol}, 
\begin{multline}\label{eq.v-cv.W}
    \fiint_W|v(X,t)-(v)_W|\1_{4Q}dXdt\lesssim\ell(W)\fiint_{2W\cap\om}(|\nabla v|+|  G|)\1_{4Q}dXdt\\
 \lesssim \wt N(\delta|\nabla v|\1_{4Q})(y,s)+ \wt\C(\delta|  G|\1_{4Q})(y,s)
\end{multline}
thanks to property (2) above. For each $i\in\set{2,\cdots,M}$, since $W_{i-1}$ and $W_i$ are adjacent in the Harnack chain, we can find a parabolic cube $S_i\subset W_i\cup W_{i-1}$ such that both $S_i\cap W_{i-1}$ and $S_i\cap W_i$ have measures comparable to $\ell(W_i)^{n+2}$. By the triangle inequality, 
\begin{multline*}
     \abs{(v)_{W_{i-1}}-(v)_{W_i}}\\
    \le \abs{(v)_{W_{i-1}}-(v)_{S_i\cap W_{i-1}}}+ \abs{(v)_{S_i\cap W_{i-1}} -(v)_{S_i}}
    +\abs{(v)_{S_i}-(v)_{S_i\cap W_{i}}}+\abs{(v)_{S_i\cap W_i}-(v)_{W_i}}.
\end{multline*}
We apply Lemma~\ref{lem.Pnc_sol} to each term, for example, 
\begin{multline*}
    \abs{(v)_{W_{i-1}}-(v)_{S_i\cap W_{i-1}}}
    =\abs{\fiint_{S_i\cap W_{i-1}}(v-(v)_{W_{i-1}})dXdt}
    \lesssim \fiint_{W_{i-1}}|v-(v)_{W_{i-1}}|dXdt\\
    \lesssim\ell(W_{i-1})\fiint_{2W_{i-1}}(|\nabla v|+|  G|)dXdt,
\end{multline*}
and we obtain that 
\[
 \abs{(v)_{W_{i-1}}-(v)_{W_i}}\lesssim \fiint_{2W_{i-1}}(|\nabla v|+|  G|)dXdt +\fiint_{2W_{i}}(|\nabla v|+|  G|)dXdt.
\]
By properties (2) and (3) above, we have that 
\[
\sum_{i=2}^N|(v)_{W_{i-1}}-(v)_{W_i}|\lesssim M\br{\wt N(\delta |\nabla v|\1_{4KQ})(y,s)+\wt\C(\delta  G\1_{4KQ})(y,s)}.
\]
This together with \eqref{eq.v-cv.sp},\eqref{eq.v-cv.W}, and property (1) above, allows us to conclude that 
\[
 \fiint_{W}|v(X,t)|\1_{4Q}(X,t)dXdt\lesssim \ln\br{\frac{r}{\ell(W)}}\br{\wt N(\delta |\nabla v|\1_{4KQ})(y,s)+\wt\C(\delta  G\1_{4KQ})(y,s)}.
\]
We now return to \eqref{eq.v-cv.sum}. Observe that for given $(y,s)\in\pom$, there is a uniformly finite number of $W\in\mathcal W$ satisfying $(y,s)\in 10 W$ and $\ell(W)=2^{-k}r$ for each $k\in\mathbb{Z}$. The condition $W\cap 4Q\neq\emptyset$ entails in additional that $k\ge -4$ for those $W$.  Therefore, 
\[
\sum_{\substack{W\in \mathcal W, (y,s)\in 10W\\ W\cap 4Q\neq\emptyset}}\ell(W)\ln\br{\frac{r}{\ell(W)}}\lesssim\sum_{k=-4}^\infty k2^{-k}r\lesssim r,
\]
which gives that 
\[
\iint_{4Q\cap\om}|v|dXdt \lesssim r\int_{(y,s)\in\pom}\br{\wt N(\delta |\nabla v|\1_{4KQ})(y,s)+\wt\C(\delta  G\1_{4KQ})(y,s)} d\sigma(y,s),
\]
and hence \eqref{eq.v-Nv} follows.

Some additional observations are needed for estimating the 3 terms from \eqref{eq.I0split}. Notice that all these terms involve derivatives of $\vp$, and so the regions of integrals are all contained in $J_{2r}\cap\om\setminus J_{\frac32r}$. 
%Since $  G=  G\1_Q$, the function $v-c_v$ satisfies the homogeneous equation $L^*(v-c_v)=0$ in $2Q\cap\om\setminus\frac32 Q$. 
We cover the region $J_{2r}\cap\om\setminus J_{\frac32r}$ by a finite collection of parabolic cubes $Q_i$ of sidelength $\ell(Q_i)\approx r$ such that 
\begin{enumerate}[(i)]
    \item either $Q_i$ is centered at $\pom$ and $4Q_i'\cap J_r=\emptyset$,
    \item or $4Q_i\subset \frac52Q\cap\om\setminus J_r$.
\end{enumerate}
Note that $v$ satisfies the homogeneous equation $L^*v=0$ in $\om\setminus J_{r}$ and has 0 Neumann data on $\pom$. Applying Lemma~\ref{lem.MoserNeu} to $v$ on $Q_i$ of type $(i)$ and applying the usual interior Moser estimate on $Q_i$ of type $(ii)$, we get that 
\begin{equation}\label{eq.vL2}
    \br{\fiint_{J_{2r}\cap\om\setminus J_{\frac32r}}|v|^2dXdt}^{1/2}\le \sup_{J_{2r}\cap\om\setminus J_{\frac32r}}|v|\lesssim r^{-n-2}\iint_{4Q\cap\om\setminus J_r}|v|dXdt.
\end{equation}
Combining \eqref{eq.vL2} and \eqref{eq.v-Nv}, one has that 
\begin{multline}\label{eq.vL2-fnl}
     \br{\fiint_{J_{2r}\cap\om\setminus J_{\frac32 r}}|v|^2dXdt}^{1/2}\lesssim r^{-n-1}\br{\|\wt N(\delta|\nabla v|\1_{4KQ})\|_{L^1(\pom)}+\|\wt\C(\delta   G\1_{4KQ})\|_{L^1(\pom)}}
     \\\lesssim r^{-\frac{n+1}{p'}}\br{\|\wt N(\delta|\nabla v|\1_{4KQ})\|_{L^{p'}(\pom)}+\|\wt\C(\delta   G\1_{4KQ})\|_{L^{p'}(\pom)}}\lesssim r^{-\frac{n+1}{p'}},
\end{multline}
where we have used in the last inequality the assumption that \wPNq$^{L^*}$ holds and \eqref{locG}.
\medskip

Now we are ready to estimate the 3 terms from \eqref{eq.I0split}. 
We choose $c_u=\fiint_{J_{8r}\cap \om}u\,dXdt$. 
Since $u-c_u$ is also a solution to $Lu=0$ in $Q_0\cap\om$ with zero Neumann data on $Q_0\cap \pom$, we have by Lemma~\ref{lem.MoserNeu} and  Lemma~\ref{lem.Pnc_sol} (recall $J_{Kr}\subset Q_0$) that 
\begin{equation}\label{eq.supu-cu.1}
    \sup_{J_{2r}\cap\om}|u-c_u|\le C\fiint_{J_{8r}\cap\om}|u-c_u|dXdt\le C\,r\fiint_{J_{16r}\cap\om}|\nabla u|dXdt.
\end{equation}
For $I_{00}$, we use  \eqref{eq.supu-cu.1}, \eqref{eq.vL2-fnl}, and the bound $|\dr_t\vp|\lesssim r^{-2}$ in $J_{2r}$, to obtain that 
\[
|I_{00}|\lesssim r^{-2}\sup_{J_{2r}\cap\om}|u-c_u|\iint_{J_{2r}\cap\om\setminus J_{\frac32r}}|v|dXdt\lesssim r^{-\frac{n+1}{p}}\fiint_{J_{16r}\cap\om}|\nabla u|dXdt.
\]
For $I_{01}$, we use H\"older inequality, \eqref{eq.vL2-fnl} and the bound on $|\nabla\vp|$ to obtain that 
\[
|I_{01}|\lesssim  r^{-\frac{n+1}{p}}\br{\fiint_{J_{2r}\cap\om}|\nabla u|^2dXdt}^{1/2}.
\]
Finally, for $I_{02}$, we use \eqref{eq.supu-cu.1}, the bound on $|\nabla\vp|$, Lemma~\ref{L:Caccio} (as well as the decomposition of the region $J_{2r}\cap\om\setminus J_{\frac2r}$ into parabolic cubes of types $(i)$ and $(ii)$ as before), and \eqref{eq.vL2-fnl} to get that 
\[
|I_{02}|\lesssim r^{n+2}\fiint_{J_{16r}\cap\om}|\nabla u|dXdt \iint_{J_{2r}\cap\om\setminus J_{\frac32r}}|\nabla v|dXdt\lesssim r^{-\frac{n+1}{p}}\fiint_{J_{16r}\cap\om}|\nabla u|dXdt.
\]
The estimate \eqref{loc2} follows as desired, and we can take $K=16$.
\end{proof}

We let the reader check that Theorem~\ref{thm.NtoLoc} follows from Lemmas~\ref{lem.N+D=wPNR}, \ref{lem.wPNR=wPN}, and \ref{LemmaM}.

\section{Proof of Theorem~\ref{MT}}\label{S5}

In this section, we present an application of the localization estimate that we have established in Theorem~\ref{thm.NtoLoc} and prove Theorem \ref{MT}.
Let 
$$\mathcal O=\{(x,x_n): x_n>\varphi(x)\},\qquad\mbox{for }x\in\R^{n-1},$$
where $\varphi$ is a Lipschitz function in variables $x$. As before $\Omega=\mathcal O\times\R$.

We assume that, for some $p\in(1,\infty)$, both \Np$^L$ $ and $ \Dq$^{L^*}$ are solvable. 
We shall establish the end-point $p=1$ Hardy space bound of the form
\begin{equation}\label{AE}
\|\tilde N(\nabla u)\|_{L^1(\partial\Omega)}\le C,
\end{equation}
for all $u$ that solve $Lu=0$ in $\Omega$ with Neumann boundary data $g$, where $g$ is any
$L^\infty$ atom, i.e.,  
$$\mbox{supp }g\subset Q_r(X,t)\cap \partial\Omega\quad\mbox{for some $(X,t)\in \partial\Omega$ and $r>0$ },
\|g\|_{L^\infty}\le r^{-n-1},\quad\int_{\partial\Omega}g=0.$$
This bound implies solvability of the Neumann boundary value problem for $p=1$ with Hardy space data. Then by interpolation we get for all $1<q<p$ solvability of the \Nq$^L$, which proves Theorem \ref{MT}. Here we interpolate the sublinear functional
$$T:g\mapsto \tilde N(\nabla u),$$
using the real interpolation method. Thanks to \eqref{AE} we know that $T$ is bounded from $\hbar^1_{ato}(\partial\Omega)\to L^1(\partial\Omega)$, where $\hbar^1_{ato}(\partial\Omega)$ is the atomic
Hardy space consisting of functions of the form $\sum_i\lambda_ig_i$ such that $\sum_i|\lambda_i|<\infty$ and each $g_i$ is an $L^\infty$ atom as defined above. $T$ is also bounded from $L^p(\partial\Omega)\to L^p(\partial\Omega)$ since we assume that \Np$^L$ is solvable. It then follows that $T$ is bounded as a functional on $L^q(\partial\Omega)$ for all $1<q<p$.
\medskip

Following Brown \cite{B}, 
by rescaling and translation, we may assume that the atom $g$ is supported in a parabolic boundary cube
$Q_1(0,0)\cap \partial\Omega$ and therefore has zero average and $|g|\le 1$. 
Let $u$ be the energy solution of the Neumann boundary value problem in $\Omega$ with datum $g$.
We aim to prove \eqref{AE} for some $C>0$ independent of $g$. 

The assumption of solvability we have made implies that for some $p>1$ we have the bound \eqref{loc1}
on the portion of boundary where $g$ vanishes on $2Q\cap \partial\Omega$ as follows from Theorem~\ref{thm.NtoLoc}. In particular $g$ vanishes for all $t<-1$ and therefore $u$ is a constant for all $t<-1$. Without loss of generality we may assume that this constant is zero. 
Because $\partial^A_\nu u=0$ on the portion of boundary $\partial\Omega\setminus \overline{Q_1(0,0)}$,
we may consider an even reflection $\tilde u$ across the boundary defined by \eqref{eqdef.reflec},
% $$\tilde u(x,x_n,t)=\begin{cases}
% u(x,x_n,t),&\qquad\mbox{for }x_n\ge \varphi(x),\\
% u(x,2\varphi(x)-x_n,t),&\qquad\mbox{for }x_n< \varphi(x),\end{cases}$$
which extends $u$ to the set $\R^{n+1}\setminus \overline{Q_1(0,0)}$. This extension $\tilde u$ solves on this domain a parabolic PDE
$$\partial_t\tilde u-\divg(\tilde A\nabla\tilde u)=0,$$
where $\tilde A=A$ on $\Omega$ and the coefficient matrix $\tilde A$ on the set $x_n<\varphi(x)$ depends only on $A(x,2\varphi(x)-x_n,t)$ and $\nabla\varphi$, making the resulting matrix bounded and uniformly elliptic on the set $\R^{n+1}\setminus \overline{Q_1(0,0)}$.
We claim that $\tilde u$ satisfies the bound
\begin{equation}\label{decay}
|\tilde u(X,t)|\le \|(X,t)\|^{-n-\eta},\qquad \mbox{for all $(X,t)$ with }\|(X,t)\|\ge 10,
\end{equation}
for some small $\eta>0$. To prove this we first establish that $\tilde u$ is bounded away from $Q_1(0,0)$.
This follows from the assumption that the $L^p$ Neumann problem is solvable. In particular we have that
$\|\tilde{N}(\nabla u)\|_{L^p(\partial\Omega)}<\infty$. \medskip

Let  $(X,t)$ be any point such that $J_1(X,t)\cap Q_1(0,0)=\emptyset$.  We claim that then we have
\begin{equation}\label{eq5.3}
\fiint_{J_{3/4}(X,t)}|\tilde u|\le C \|\tilde{N}(\nabla u)\|_{L^p(\partial\Omega)},
\end{equation}

and hence by interior H\"older regularity of $\tilde u$
\begin{equation}\label{eq5.3a}
\sup_{J_{1/2}(X,t)}|\tilde u| \le C \|\tilde{N}(\nabla u)\|_{L^p(\partial\Omega)}.
\end{equation}
To see \eqref{eq5.3} consider points with $\delta(X,t)\le O(1)$ and $t\le O(1)$. For these clearly
\begin{equation}\label{eq5.4}
\fiint_{J_{1}(X,t)}|\nabla \tilde u|\lesssim \fint_{Q_1(x,\varphi(x))\times(t-1,t)\cap \partial\Omega}\tilde N(\nabla u)d\sigma\le C \|\tilde{N}(\nabla u)\|_{L^p(\partial\Omega)},
\end{equation}
when $p>1$. Then by interior Poincar\'e's inequality for $\tilde u$ (Lemma \ref{lem.Pnc_sol})
$$\fiint_{J_{3/4}(X,t)}|\tilde u-c_{\tilde{u}}|\le C \|\tilde{N}(\nabla u)\|_{L^p(\partial\Omega)},$$
where $c_{\tilde{u}}$ denotes the average of $\tilde{u}$ in $J_{3/4}(X,t)$. Clearly, if we can take $c_{\tilde{u}}=0$ the claim would follow. We know however that this holds for $t<-1$. We claim that simple geometric considerations imply that there exist $N=N(n)$ and a chain of regions $J_1(X_1,t_1), J_1(X_2,t_2)$,$\dots$, $J_1(X_k,t_k)$
with $k\le N$ such that
\begin{itemize}
\item regions $J_1(X_i,t_i)$, $J_1(X_{i+1},t_{i+1})$ overlap but 
$J_{3/4}(X_i,t_i)$, $J_{3/4}(X_{i+1},t_{i+1})$ do not for $i=1,2,\dots,k-1$.
\item $J_1(X_i,t_i)\cap Q_1(0,0)=\emptyset$ for $1\le i\le k$,
\item either $t_i=t_{i+1}$ or $X_i=X_{i+1}$ for each $i=1,2,\dots,k-1$,
\item $t_1<-2$ and $(X_k,t_k)=(X,t)$.
\end{itemize}
The last condition implies that the average of $\tilde u$ over $J_{3/4}(X_1,t_1)$ vanishes. 
We claim that the third condition implies that the difference of averages of $\tilde u$ between
$J_{3/4}(X_i,t_i)$, $J_{3/4}(X_{i+1},t_{i+1})$ can be controlled in terms of an integral of $|\nabla\tilde{u}|$
over the convex hull of the union of $J_1(X_i,t_i)$ and $J_1(X_{i+1},t_{i+1})$ and hence by considerations
analogous to \eqref{eq5.4} ultimately again by $C \|\tilde{N}(\nabla u)\|_{L^p(\partial\Omega)}$.
Chaining the difference of averages implies a bound for the average of $\tilde u$ over 
$J_{3/4}(X_k,t_k)=J_{3/4}(X,t)$ by at most $CN \|\tilde{N}(\nabla u)\|_{L^p(\partial\Omega)}$, hence proving our claim.

To see that the difference of averages of $\tilde u$ between
$J_{3/4}(X_i,t_i)$, $J_{3/4}(X_{i+1},t_{i+1})$ can be controlled in terms of an integral of $|\nabla\tilde{u}|$
consider first the case when $t_i=t_{i+1}$. Then the statement is trivial and follows by expressing the difference
of two values of $\tilde u$ at two points in each regions with the same time coordinate by the fundamental theorem of calculus by integrating $\nabla\tilde{u}$ along the line segment joining these two points.

When $X_i=X_{i+1}$ but $t_i\ne t_{i+1}$ the argument is a bit more involved and requires the use of the PDE
$\tilde{u}$ satisfies. This is done in full detail in \cite{DG} (see the calculation after (6.23) estimating the average on two parabolic cubes in term of $\nabla u$. We use the same idea in the proof of Lemma \ref{lem.Pnc_sol}, see \eqref{eq.avgt-t'} and the subsequent calculation).  Hence, indeed \eqref{eq5.3a} holds
for all points with $\delta(X,t)\le O(1)$ and $t\le O(1)$. Boundedness for the remaining points (i.e. those where
either $\delta(X,t)$ is large or $t$ is large) follows by the maximum principle (which must hold even though
our domain is unbounded in the $x_n$ variable, since the fact that $\tilde{u}$ is an energy solution implies a weak decay of $\tilde{u}$ when $|x_n|\to\infty$). Hence, $\tilde{u}$ is bounded on any complement of an enlargement of $Q_1(0,0)$.\medskip

Consider now a smooth cutoff function $\eta$ which vanishes on $Q_1(0,0)$ and is equal to $1$ outside of $Q_2(0,0)$. It follows that $\eta\tilde{u}$ is bounded on $\R^n\times\R$ and hence for any $c>0$ and any bounded time interval $(a,b)$ we have that
$$\int_a^b\iint_{\R^n}|(\eta\tilde{u})(X,t)|^2e^{-c|X|^2}dX\,dt<\infty,$$
which then implies the integrability also on the interval $(-\infty,b)$ as $\tilde u$ vanishes when $t<-1$.
Recall  Aronson’s bound for the fundamental 
solution (\cite{Aro68}), denoted by $E(X,t,Y,s)$:
\begin{equation}\label{KE}
|E(X,t,Y,s)|\le C\chi_{(0,\infty)}(t-s)(t-s)^{-n/2}e^{-c|X-Y|^2/(t-s)},
\end{equation}
for some $c>0$. Also $\int_{\R^n} E(X,t,Y,s)dY=1$ for all $s\in\R$ and $(X,t)\in\R^n\times\R$.

It follows that $(\eta\tilde u)$ can be written as
\begin{multline*}
\eta\tilde u(X,t)=\iint_{\R^n\times\R}E(X,t,Y,s)\left[\tilde u(Y,s)\partial_s\eta(Y,s)-\langle \tilde A(Y,s)\nabla\tilde u(Y,s),\nabla\eta(Y,s)\rangle \right]dYds\\+\iint_{\R^n\times\R}
\tilde u(Y,s)\langle \tilde A(Y,s)\nabla \eta(Y,s),\nabla_YE(X,t,Y,s)\rangle dYds.
\end{multline*}
Hence using \eqref{KE} and Caccioppoli's inequality for the term involving $\nabla E$, it follows that for some small $c>0$
$$|\tilde u(X,t)|\le Ce^{-c|X|^2},\qquad \mbox{for all $(X,t)$ with }\|(X,t)\|\ge 4\mbox{ and }t<20.$$
Using the divergence theorem we get at the time $t=10$ that 

\begin{multline*}
\int_{\mathcal O} u(X,10)dX=\int_{\mathcal O} [u(X,10)-u(X,-1)]dX=\int_{-1}^{10}\int_{\mathcal O}\partial_t u(X,t)dXdt\\=
\int_{-1}^{10}\int_{\mathcal O}\divg(A\nabla u(X,t))dXdt=
\int_{-1}^{10}\int_{\partial\mathcal O}\partial^A_\nu u(Y,s)  d\sigma ds=\int_{\partial\Omega}g=0,    
\end{multline*}
from which we deduce  $\int_{\R^n} \tilde u(X,10)dX=0$, as the extension $\tilde u$ is an even reflection of $u$.
From this the claim \eqref{decay} follows by applying the following lemma after substituting $t\mapsto t+10$ (c.f. Lemma 3.13 of \cite{B}):

\begin{lemma}\label{lem.decay} Suppose $f:\R^n\to\R$ satisfies for some small $\eta>0$ the decay estimate and the averaging condition
$$|f(X)|\le (1+|X|)^{-n-2\eta}\qquad\mbox{and} \qquad \int_{\R^n}f(X)dX=0.$$
Then the solution of the initial value problem 
$$\partial_t u-\divg(\tilde A\nabla u)=0\qquad\mbox{in }\R^n\times(0,\infty),\qquad u(\cdot,0)=f(\cdot),$$
satisfies $|u(X,t)|\le (1+\|(X,t)\|)^{-n-\eta}$.
\end{lemma}

\begin{proof}[Proof of Lemma~\ref{lem.decay}] By the maximum principle clearly $|u|\le 1$ from the initial bound. Hence showing the bound 
$|u(X,t)|\le \|(X,t)\|^{-n-\eta}$ for $t>0$ will suffice. Again using the kernel $E$ we can write
\begin{multline*}
u(X,t)=\int_{\R^n}E(X,t,Y,0)f(Y)dY
=\int_{\{Y:|Y|<\|(X,t)\|/10\}}[E(X,t,Y,0)-E(X,t,0,0)]f(Y) dY\\
+E(X,t,0,0)\int_{\{Y:|Y|<\|(X,t)\|/10\}}f(Y)dY+\int_{\{Y:|Y|\ge\|(X,t)\|/10\}}E(X,t,Y,0)f(Y)dY\\
:=u_1(X,t)+u_2(X,t)+u_3(X,t).
\end{multline*}
We start with $u_3$. 
Since $\int_{\R^n} E(X,t,Y,s)dY=1$ and we are integrating only over the region where 
$|f|\le \|(X,t)\|^{-n-\eta}$ the bound is immediate. For $u_2$: because the average of $f$ is zero we have
\begin{multline*}
\left|\int_{\{Y:|Y|<\|(X,t)\|/10\}}f(Y)dY\right|=\left|\int_{\{Y:|Y|\ge\|(X,t)\|/10\}}f(Y)dY\right|\\\le
\int_{\{Y:|Y|\ge\|(X,t)\|/10\}}(1+|Y|)^{-n-2\eta}dY\le C\|(X,t)\|^{-\eta},
\end{multline*}
and from this the bound for $u_2$ follows from the upper bound for $E$ (c.f. \eqref{KE}).
Finally for $u_1$: using H\"older continuity of $E$ in the adjoint variables and the upper bound for it from \eqref{KE} we have
\begin{multline*}
|u_1(X,t)|\le \sup_{\{(Y,s):\|(Y,s)\|<2\|(X,t)\|\}}E(X,t,Y,s)\int_{\R^n}\left(\frac{|Y|}{\|(X,t)\|}\right)^\eta
(1+|Y|)^{-n-2\eta}dY\\
\le C \|(X,t)\|^{-n-\eta},
\end{multline*}
where $\eta$ is taken to be at most the H\"older exponent of modulus of continuity of $E$.
\end{proof}

We are now ready to start the actual proof of estimate \eqref{AE}.

\begin{proof}[Proof of \eqref{AE}]
We start with estimating $\tilde N(\nabla u)$ in a neighbourhood of $Q_1(0,0)\cap \partial\Omega$.
By H\"older's inequality  and $L^p$ Neumann solvability we have that
\begin{multline}\label{eq5.5}
\|\tilde N(\nabla u) \|_{L^1(Q_{16}(0,0)\cap \partial\Omega)}\le |Q_{16}(0,0)\cap\partial\Omega|^{1/p'}\left(\int_{Q_{16}(0,0)\cap\partial\Omega} \tilde N(\nabla u)^p d\sigma\right)^{1/p}\\
\lesssim  |Q_{16}(0,0)\cap\partial\Omega|^{1/p'} \left(\int_{\partial\Omega}|g|^p\right)^{1/p}\le  
|Q_{16}(0,0)\cap\partial\Omega|^{1/p'} |Q_{16}(0,0)\cap\partial\Omega|^{1/p}\\=|Q_{16}(0,0)\cap\partial\Omega|\le C, 
\end{multline}
since the volume of a surface cube of size $16$ on a Lipschitz graph is bounded by a uniform constant that only depends on the Lipschitz constant. 

Consider now a sequence of annuli $S_k$ partitioning the complement of the set $Q_{16}(0,0)\cap\partial\Omega$ in $\partial\Omega$. That is let
$$S_k=\{(X,t)\in \partial\Omega: 2^k<\|(X,t)\|\le 2^{k+1}\},\qquad\mbox{for } k=4,5,\dots.$$
Clearly, as the volume of $S_k$ is proportional to $2^{k(n+1)}$ there exists $N$, independent of $k$ and only depending on $n$, the Lipschitz norm of $\varphi$ and $c>0$, such that:
\begin{itemize}
\item For $r_k=c2^k$ there exist a cover of $S_k$ by $N$ boundary cubes $Q_{r_k}^{k,j}\cap\partial\Omega$
of the same size $r_k$, $j=1,2,\dots, N.$
\item $d(Q_{8r_k}^{k,j}\cap\partial\Omega,Q_1(0,0)\cap\partial\Omega)\ge c2^k$.
\end{itemize}

At this point we depart in our argument significantly from \cite{B}. Thanks to Theorem \ref{thm.NtoLoc} we have 
for each $Q_{r_k}^{k,j}$ and a truncated version $\tilde N^{r_k}$ (defined using truncated cones $\Gamma^{r_k}$) of the nontangential maximal function
\begin{multline}\label{eq5.6}
\|\tilde N^{r_k}(\nabla u) \|_{L^p(Q_{r_k}^{k,j}\cap \partial\Omega)}\le 
\|\wt N(|\nabla u|\1_{2Q_{r_k}^{k,j}})\|_{L^p(\partial \Omega)} \\\leq C r_k^{(n+1)/p}\left ( \fiint_{4Q_{r_k}^{k,j}\cap \om} |\nabla u|^2 dXdt \right)^\frac12\le C r_k^{-1+(n+1)/p}\left ( \fiint_{8Q_{r_k}^{k,j}\cap \om} |u|^2 dXdt \right)^\frac12.
\end{multline}
Here we have used the fact that $\partial^A_\nu u=0$ on $8Q_{r_k}^{k,j}\cap \om$ which allows us to use boundary Caccioppoli (Lemma~\ref{L:Caccio}) in the last estimate. 
We may now use the estimate \eqref{decay} to bound \eqref{eq5.6} further. It follows that 
\begin{equation}\label{eq5.7}
\|\tilde N^{r_k}(\nabla u) \|_{L^p(Q_{r_k}^{k,j}\cap \partial\Omega)}\le 
C r_k^{-1+(n+1)/p}r_k^{-n-\eta}=Cr_k^{-1-n+(n+1)/p-\eta}=Cr_k^{-(n-1)/p'-\eta}.
\end{equation}
To obtain an $L^1$ bound on each $S_k$ we use H\"older's inequality and sum over $j=1,2,\dots,N$:
\begin{multline}\label{eq5.8}
\|\tilde N^{r_k}(\nabla u) \|_{L^1(S_k)}\le \sum_{j=1}^N \|\tilde N^{r_k}(\nabla u) \|_{L^1(Q_{r_k}^{k,j}\cap \partial\Omega)}\\\le CN |Q_{r_k}^{k,j}\cap \partial\Omega|^{1/p'}\|\tilde N^{r_k}(\nabla u) \|_{L^p(Q_{r_k}^{k,j}\cap \partial\Omega)}\lesssim  r_k^{(n+1)/p'}
r_k^{-(n-1)/p'-\eta}\sim r_k^{-\eta}.
\end{multline}
It remains to estimate the part of $\|\tilde N(\nabla u) \|_{L^1(S_k)}$ away from the boundary. To that end let us consider any point $(Y,s)\in\Gamma(X,t)$ for some $(X,t)\in S_k$ such that $\delta(Y,s):=d((Y,s),\partial\Omega)\ge r_k$.
Using the interior Caccioppoli's inequality we have that
$$\left(\fiint_{Q_{\delta(Y,s)/2}(Y,s)}|\nabla u|^2\right)^{1/2}\lesssim r_{k}^{-1}\sup_{Q_{3\delta(Y,s)/4}(Y,s)} |u|
\le Cr_{k}^{-1} r_k^{-n-\eta},
$$
since points in $Q_{3\delta(Y,s)/4}(Y,s)$ have distance to $Q_1(0,0)\cap\partial\Omega$ at least $cr_k$. It follows that 
the contribution of this term to $\tilde N(\nabla u)(X,t)$ is at most $C r_k^{-1-n-\eta}$ and as the size of 
$S_k$ is $\sim r_k^{n+1}$ the $L^1$ norm of its contribution is of the same order as the term \eqref{eq5.8}, i.e., $Cr_k^{-\eta}$.

This estimate, together with \eqref{eq5.5} and \eqref{eq5.8}, yield:
\begin{multline}\label{eq5.9}
\|\tilde N(\nabla u) \|_{L^1(\partial\Omega)}\\\le \|\tilde N(\nabla u) \|_{L^1(Q_{16}(0,0)\cap \partial\Omega)}
+\sum_{k\ge 4}\|\tilde N(\nabla u) \|_{L^1(S_k)}\le C+C\sum_{k\ge 4}2^{-k\eta}\le \tilde{C}<\infty.
\end{multline}
Hence claim \eqref{AE} follows. 
\end{proof}

\medskip

\bibliographystyle{alpha}
%\bibliography{reference}

\end{document}